\documentclass[12pt,epsfig,amsfonts]{amsart} 
\usepackage{amsmath,amsthm,amssymb,amscd,epsfig}
\usepackage{graphicx}

\newtheorem*{theorem}{Theorem}
\newtheorem*{prop}{Proposition}
\newtheorem{lemma}{Lemma}[section]
\newtheorem{sublemma}[lemma]{Sublemma}
\newtheorem{cor}{Corollary}

\begin{document}
\author{Yong Moo Chung and Hiroki Takahasi}
\address{Department of Applied Mathematics, Hiroshima University,
Higashi-Hiroshima 739-8527, JAPAN} \email{chung@amath.hiroshima-u.ac.jp}
\address{FIRST, Aihara Innovative Mathematical Modelling Project,
Japan Science and Technology Agency, Institute of Industrial Science, University of Tokyo, Tokyo
153-8505, JAPAN \\
{\it Current address:} Department of Electronic Science and Engineering,
 Graduate School of Engineering, Kyoto University,
Kyoto 606-8501, JAPAN} \email{takahasi.hiroki.7r@kyoto-u.ac.jp}
\title[Large deviation principle
for Benedicks-Carleson  quadratic maps] {Large deviation
principle for \\
Benedicks-Carleson quadratic maps}

\begin{abstract}
Since the pioneering works of Jakobson and Benedicks $\&$
Carleson and others, it has been known that a positive measure set of quadratic
maps admit invariant probability measures absolutely continuous
with respect to Lebesgue. These measures allow one to
statistically predict the asymptotic fate of Lebesgue almost every
initial condition. Estimating
fluctuations of empirical distributions 
before they settle to equilibrium requires a fairly good control
over large parts of the phase space. We use the sub-exponential
slow recurrence condition of Benedicks $\&$ Carleson to build
induced Markov maps of arbitrarily small scale and associated towers, to which the
absolutely continuous measures can be lifted. These
various lifts together enable us to obtain a control of
recurrence that is sufficient to establish a level 2 large deviation
principle, for the absolutely continuous measures.
This result encompasses dynamics far from equilibrium, and
thus significantly extends presently known local large deviations
results for quadratic maps.
\end{abstract}
\maketitle

\section{Introduction}
Let $X=[-1,1]$, and let $f_a\colon X\circlearrowleft$ be the
quadratic map given by $f_ax=1-ax^2$, where $0< a\leq2$. The
abundance of parameters in this family for which ``chaotic
dynamics" occur has been known since the pioneering works of
Jakobson \cite{Jak81} and Benedicks $\&$ Carleson
\cite{BC85,BC91}: there exists a set of $a$-values near $2$ with
positive Lebesgue measure for which the corresponding $f=f_a$
admits an invariant probability measure $\mu$ that is absolutely
continuous with respect to Lebesgue (acip). By a classical
theorem, for Lebesgue a.e. $x$ the empirical distribution
$\delta_x^n=(1/n)\sum_{i=0}^{n-1}\delta_{f^ix}$ converges weakly to
$\mu$.
The theory of large deviations aims to
provide exponential bounds on the probability that $\delta_x^n$
stays away from $\mu$.

Large deviations questions have been addressed for various
stochastic processes (see e.g. \cite{DonVar75,DonVar76}). 
For dynamical
systems, one cannot expect a full large deviation principle
without strong assumptions \cite{ComRiv11,K,OP, PrzRiv11,Tak82,Tak85,You90}. For the quadratic
map, the enemy is the critical point $x=0$. Up to now, only local
large deviations results are known \cite{KelNow92,MN,ReyYou08},
and a full result which encompasses dynamics far from equilibrium
is still unknown.
 Our aim here is to provide a simple set of
conditions satisfied on a positive measure set in parameter space
and to show that when these conditions are met, a full large
deviation principle holds.

We formulate our conditions as follows: Let
 $\lambda=\frac{9}{10}\log2$ and
$\alpha=\frac{1}{100}$.
\begin{itemize}
\item[(A1)] $f=f_a$ where $a$ is sufficiently near $2$;
\item[(A2)] $|Df^n(f0)|\geq e^{\lambda n}$ for every $n\geq0;$
\item[(A3)] $|f^n0|\geq e^{-\alpha\sqrt{n}}$ for every $n\geq1$;
\item[(A4)] $f$ is topologically mixing on $[f^20,f0]$.
\end{itemize}

Benedicks $\&$ Carleson \cite{BC91} proved the 
the abundance of parameters near $2$ for which (A2) holds.
The abundance
of parameters for which (A3) holds was proved by Benedicks
$\&$ Young \cite{BY92}, and previously by Benedicks $\&$ Carleson
\cite{BC85} under slightly different hypotheses. 
For their parameters, (A4) holds (see \cite[Lem. 2.1]{You92}).
The parameter
sets they constructed have $2$ as a full Lebesgue density point.
Hence, given $a_0<2$ arbitrarily near $2$, there is a set
$A\subset[a_0,2]$ with positive Lebesgue measure such that (A2)-(A4) 
hold for all
$a\in A$.

In what follows, we assume (A1)-(A4) for $f=f_a$.
Then $f$
admits an acip $\mu$.
Let $\mathcal M$ denote the space of Borel probability measures on
$X$ endowed with the topology of weak convergence. 
Let $\mathcal M_f$ denote the set of $f$-invariant Borel
probability measures. Define a Lyapunov exponent of
$\nu\in\mathcal M_f$ by $\lambda(\nu)=\int\log |Df|d\nu$. This is
 strictly positive for any $\nu\in\mathcal
M_f$ \cite{BruKel98,NowSan98}. Let $h(\nu)$ denote the entropy of
$\nu$, and define a free energy function $F\colon \mathcal
M\to\mathbb R\cup\{-\infty\}$ by
\[F(\nu)=\begin{cases} &h(\nu)-\lambda(\nu)\ \
\text{if}\ \ \nu\in\mathcal M_f;  \\
&-\infty \ \ \ \ \ \ \ \ \ \ \ \ \text{otherwise.}\end{cases}\]
By Ruelle's inequality \cite{Rue78}, $F(\nu)\leq0$ and the equality holds only
if $\nu=\mu$ \cite{Led81}.
It is known \cite{BruKel98} that the Lyapunov exponent is not lower semi-continuous,
and so $-F$ may not be lower semi-continuous.
Hence we introduce its lower-semi-continuous regularization $I\colon\mathcal M\to[0,\infty]$
by $$I(\nu)= - \inf_{\mathcal G} \sup \{ F(\xi)\colon \xi\in \mathcal G\},$$
where the infimum is taken over all neighborhoods $\mathcal
G$ of $\nu$ in $\mathcal M$.
Let $\delta_x^n=(1/n)\sum_{i=0}^{n-1}\delta_{f^ix}$,
where $\delta_{f^ix}$ is the Dirac measure at $f^ix$.
Let $\log0=-\infty$.

\begin{theorem}
Let $f=f_a$ satisfy {\rm (A1)}-{\rm (A4)}. Then the large deviation
principle holds for $(f,\mu)$ with $I$ the rate function, namely
\begin{equation*}\label{low}
\varliminf_{n\to\infty}\frac{1}{n} \log \mu\{x\in X\colon
\delta_x^n\in \mathcal G\}\geq-\inf\{I(\nu)\colon {\nu\in \mathcal
G}\}\end{equation*} for any open set $\mathcal G\subset\mathcal M$, and
\begin{equation*}\label{up}
\varlimsup_{n\to\infty}\frac{1}{n} \log \mu\{x\in X\colon
\delta_x^n\in \mathcal K \}\leq-\inf\{I(\nu)\colon\nu\in \mathcal
K\}\end{equation*}
for any closed set $\mathcal K\subset\mathcal
M$. 
\end{theorem}

We state a corollary which follows from the Contraction Principle in large
deviations, and use it to compare our result
with the previous related ones. Let $C(X)$ denote the space of all
continuous functions on $X$. For $\varphi\in C(X)$, write
$S_n\varphi=\sum_{i=0}^{n-1}\varphi\circ f^i$ and define
$$c_\varphi=\inf_{x\in X}\varliminf_{n\to\infty}\frac{1}{n}S_n\varphi(x)
\ \text{ and } \ d_\varphi=\sup_{x\in
X}\varlimsup_{n\to\infty}\frac{1}{n}S_n\varphi(x).$$
 The
compactness of $\mathcal M_f$ implies
$c_\varphi=\min\{\nu(\varphi)\colon\nu\in\mathcal M_f\}$ and
$d_\varphi=\max\{\nu(\varphi)\colon\nu\in\mathcal M_f\}$, where
$\nu(\varphi)=\int\varphi d\nu$. We assume $c_\varphi<d_\varphi$,
for otherwise it is meaningless to consider $\varphi$.
Define $F_{\varphi}\colon[c_\varphi,d_\varphi]\to\mathbb R$ by $$F_\varphi(t)=\sup\{F(\nu)\colon \nu\in\mathcal M_f,\nu(\varphi)
=t\},$$ which is concave and so continuous on $(c_\varphi,d_\varphi).$ 

\begin{cor}\label{contract}
For all $a,b\in [c_\varphi,d_\varphi]$ such that $a<b$ we have
\[\lim_{n\to\infty}\frac{1}{n}\log \mu\left\{
a\leq\frac{1}{n}S_n\varphi\leq b\right\}=\max_{a\leq t\leq b
}F_\varphi(t).\]
\end{cor}
Our theorem is the first full large deviations result for a
positive measure set of quadratic maps, despite a large number of
papers over the past thirty years dedicated to stochastic
properties of chaotic dynamics in one-dimensional maps. Up to now,
only local results are known, which claim the
existence of the above limit in the case where $\varphi$ is
H\"older continuous and $a,b$ are near the mean $\mu(\varphi)$
 \cite{KelNow92,MN,ReyYou08}.

The next corollary follows from Varadhan's integral lemma \cite[p.137]{DZ}
and the convex duality of the Fenchel-Legendre transforms \cite[p.152]{DZ}.
\begin{cor}
For any $\varphi\in C(X)$, the limit
\[P(\varphi)=\lim_{n\to\infty}\frac{1}{n}\log \mu(e^{S_n\varphi})\]
exists. In addition, $(P,I)$ form a Legendre pair, namely the following holds:
$$P(\varphi)=\max\{\nu(\varphi)-I(\nu)\colon\nu\in\mathcal M_f\}\ \ \text{for}
\ \varphi\in C(X);$$
$$I(\nu)=\max\{\nu(\varphi)-P(\varphi)\colon\varphi\in C(X)\}\ \
\text{for}\ \ \nu\in\mathcal M_f.$$
\end{cor}


For a broad class of nonuniformly hyperbolic systems
including the quadratic maps
we treat here, {\it towers} (or {\it inducing schemes}) have been heavily used to draw their interesting properties (see e.g. 
\cite{BalVia96,BruKel98,You98,You99}).
A proof of the theorem also relies on the construction of induced Markov maps and associated towers. There is a great deal of freedom and flexibility in the construction of towers, and hence the issue is to construct a ``nice tower" that captures
relevant information of the underlying system.
This issue has been addressed in the literature, and it is sometimes referred to as the {\it liftability problem}. A less emphasized issue deals with the construction of 
{\it a family of towers for a given single system}, so that they altogether 
provide relevant information. This type of approach can be found, for instance, in \cite{BruTod09,IomTod11,PrzRiv11}, and appears to be successful when it is difficult to obtain 
necessary information just by considering a single tower. 

For our quadratic maps, towers have already been constructed (see e.g. \cite{BalVia96,BLS03,You92,You98}), for which the decay rate of the
tail of return times is exponential.
We emphasize that exponential tails do not necessarily imply the full
large deviation principle, primarily because probabilities of rare
events included in the tails are unaccounted for. For instance, for
certain Markov processes it is well known
\cite{BJV91,DonVar75,DonVar76} that exponential tails of return
times are in general not sufficient to ensure a full large
deviation principle. They only imply a local large deviation
result, which is similar to the results in
\cite{MN,ReyYou08}. A full large deviation principle for stationary processes
has been established under very strong mixing conditions \cite{Bry90,BryDem96}, which cannot be expected for dynamical systems.

In \cite{Chu11},  sufficient conditions on the ``shape" of towers
were introduced to ensure a full large deviation principle
for Lebesgue measure. However, for our quadratic
maps it is difficult to construct such ``ideal towers", apart from very
special cases (e.g. Misiurewicz maps). Therefore, we abandon
working with a single tower and instead construct various induced
Markov maps and associated towers. We use them together to
obtain an upper exponential bound, on the probability that time
averages of continuous functions stay away from their spatial
averages. We establish the large deviation principle by 
comparing this upper bound with a lower exponential one,
which is obtained directly from \cite{Chu11,Now93,You92}.
The
upper and lower large
deviation bounds were obtained in a very
general setting in \cite{You90}, and for nonuniformly expanding
maps in \cite{AraPac06,Var12}. These bounds are not comparable and hence insufficient to conclude
the large deviation principle.

A proof of the theorem is briefly outlined as follows.
 Given $d\geq1$, functions $\varphi_1,\ldots,\varphi_d$
on $X$ and $b_1,\ldots,b_d\in\mathbb R$, define
$$\overline{R}\left(\varphi_1,\ldots,\varphi_d;b_1,\ldots,b_d\right)=
\varlimsup_{n\to\infty}\frac{1}{n} \log \mu\left\{
\frac{1}{n}S_n\varphi_j\geq b_j,\ \ j=1,\ldots, d\right\},$$ and
$$\underline{R}\left(\varphi_1,\ldots,\varphi_d;b_1,\ldots,b_d\right)=
\varliminf_{n\to\infty}\frac{1}{n}\log \mu\left\{
\frac{1}{n}S_n\varphi_j> b_j,\ \ j=1,\ldots, d\right\}.$$
All our effort is dedicated to
proving the following proposition.
\begin{prop}\label{upper}
Let $f=f_a$ satisfy {\rm (A1)-(A4)}. Let $d\geq1$ and let
$\varphi_1,\ldots,\varphi_d$ be a collection of Lipschitz
continuous functions on $X$, and let $b_1,\ldots,b_d\in\mathbb R$.
For any $\varepsilon>0$ there exists $n_0\in\mathbb N$ such that
if $n\geq n_0$ then there exists 
$\sigma\in\mathcal M_f$ such that:
\begin{equation}
\label{upper1}
\frac{1}{n} \log \mu\left\{ \frac{1}{n}S_n\varphi_j\geq b_j,\
j=1,\ldots,d\right\}\leq
(1-\varepsilon^{1/5})F(\sigma)+2\varepsilon^{1/5};\end{equation}
\begin{equation}\label{upper5}
\sigma(\varphi_j) \geq b_j-\varepsilon^{1/2},\ \ \  j=1,\ldots,d.
\end{equation}
\end{prop}
It then follows that
\begin{equation}\label{upper-1}
\overline{R}\left(\varphi_1,\ldots,\varphi_d;b_1,\ldots,b_d\right)
\leq\lim_{\varepsilon\to0}\sup\left\{F(\nu)\colon\nu\in\mathcal
M_f,\ \nu(\varphi_j)\geq b_j-\varepsilon^{1/2},\ \ j=1,\ldots, d
\right\},\end{equation} 
where we let $\sup\emptyset=-\infty$ by convention. 
Meanwhile, by a result of 
\cite{Now93,You92},
the density
of $\mu$ is uniformly bounded away from zero on 
$[f^20,f0]$. Hence, the lower bound obtained in
\cite{Chu11} for Lebesgue translates into a lower
bound for $\mu$, namely
\begin{equation}\label{lower-1}\underline{R}\left(\varphi_1,\ldots,\varphi_d;b_1,\ldots,b_d\right)
\geq\sup\left\{F(\nu)\colon\nu\in\mathcal M,\ \nu(\varphi_j)>
b_j,\ \ j=1,\ldots, d \right\}.\end{equation} Observe that
the weak topology on $\mathcal M$ has a countable base 
generated by open sets of the form $ \left\{\nu\in\mathcal M
\colon\nu(\varphi_j)>b_j,\ j=1,\ldots,d \right\},$ where $d\geq1$,
each $\varphi_j$ is Lipschitz continuous and $b_j\in\mathbb R$.
Hence, (\ref{upper-1}) (\ref{lower-1}) imply the theorem.

Our strategy for the proof of the proposition is to construct a family of 
towers and use them to construct various horseshoes carrying invariant measures with the properties as in the statement. At this point, we make
important use of the sub-exponential slow recurrence condition (A3).

The rest of this paper consists of two sections. In Sect. 2 we
develop preliminary estimates and constructions. We modify the
classical binding argument and the return
time estimate \cite{BC85,BC91}, so that we can
treat an arbitrarily small $\varepsilon$. In Sect. 3 we prove the
proposition. A crucial estimate is Lemma \ref{Mar}, which roughly states
that \emph{any partition element of the tower is approximated by 
points which quickly return to the base of the tower}.
To equip our tower with this property, we construct 
an induced map on a Cantor set, which consists of points slow recurrent 
to the critical point. 
This construction is inspired by that of Benedicks \& Young for 
H\'enon-like attractors \cite{BenYou00}.

To maintain the brevity of this paper we refrain from
generalizations. Our arguments and results may be generalized
to $C^2$ Collet-Eckmann unimodal maps
with non-flat critical point, for which the recurrence of the
critical orbit is sub-exponential. It is known \cite{AM05} that almost every
stochastic quadratic map satisfies these two conditions.

\section{Preliminary estimates and constructions}
In this section, we develop preliminary estimates needed
for the proof of the proposition.
We develop a binding argument for recovering expansion
and prove a return time estimate.
Original ideas for these can be found in \cite{BC85,BC91}.
We modify them to treat an arbitrarily small $\varepsilon>0$.
We suppose that $\varepsilon$ is given, and then
choose sufficiently large integer $N$.

We use the following notations: $c_0= f0$ and
$c_n=f^nc_0$ for $n\geq1$; $|\cdot|$ is the Lebesgue measure on $X$;
for a set $A\subset X$, $d(0,A)=\inf\{|x|\colon x\in
A\}$; given a partition $\mathcal P$ of $A\subset X$
and $B\subset A$,
$\mathcal P|B=\{\omega\cap B\colon \omega\in\mathcal P\}$.

\subsection{Bounded distortion}\label{bddi}
For $n\geq1$, let
\begin{equation}\label{Theta}
D_n=\frac{1}{10}\cdot\left[\sum_{i=0}^{n-1} d_i^{-1}\right]^{-1}, \
\ \ \ \text{where} \ \ \ \ \ d_i=\frac{|c_i|}{|Df^i(c_0)|}.
\end{equation}

\begin{lemma}\label{dist}
For all $x,y\in I=[1-D_n,1]$, $$\frac{Df^n(x)}{Df^n(y)}\leq2\
\ \text{and} \ \
\left|\frac{Df^n(x)}{Df^n(y)}-1\right|\leq\frac{|x-y|}{D_n}.$$
\end{lemma}

\begin{proof}
 The first inequality would hold
if for every $0\leq j \leq n-1$ we have
\begin{equation}\label{disteq}
0\notin f^jI,\ \ \  \frac{|f^jI|}{ d(0,f^jI)}\leq \log 2 \cdot
d_j^{-1} \left[\sum_{i=0}^{n-1}d_i^{-1}\right]^{-1}.
\end{equation}
Indeed, if this is the case, then for $x,y \in I$,
\begin{align*}
\log\frac{Df^n(x)}{Df^n(y)}&\leq\sum_{j=0}^{n-1} \log
\frac{Df(f^{j}x)}{Df(f^{j}y)}
\leq\sum_{j=0}^{n-1}\frac{|f^jI|}{d(0,f^jI)}\leq\log 2.
\end{align*}
The second inequality follows from $|Df(x)|=2a|x|$ and $|D^2f(x)|=2a$.

 It is immediate to
check (\ref{disteq}) for $j=0$. The rest of the proof is
 by induction on $j$. Let $k>0$ and assume
(\ref{disteq}) for every $0\leq j < k$. Summing (\ref{disteq}) over
all $0\leq j<k$ implies $|Df^k(x)|\leq2|Df^k(y)|$ for all $x,y\in I$. Hence
\begin{equation}\label{derivative6}
|f^{k}I|\leq 2 |Df^{k}(c_0)| D_n = 2|f^kc_0|d_{k}^{-1}D_n\leq
(1/5)|c_k|,
\end{equation}
and thus $0\notin f^k I$ holds. For the second half of
(\ref{disteq}) we have
\begin{align*}
\frac{|f^kI|}{d(0,f^kI)}&\leq  \frac{2|Df^k(c_0)| D_n}{
d(0,f^kI)}=\frac{2d_k^{-1}D_n \cdot |
c_k|}{d(0,f^kI)}\leq\frac{3}{10}d_{k}^{-1} \left[\sum_{i=0}^{n-1}
d_i^{-1}\right]^{-1}.
\end{align*}
For the last inequality we have used $\frac{|c_k|}{d(0,f^kI)}\leq
3/2$ which follows from (\ref{derivative6}).

 For $0\leq
i<n$ we have $\frac{|c_i|}{d(0,f^i[x,y])}\leq 3/2$ and
$|f^i[x,y]|\leq
2|Df^i(c_0)||x-y|d_id_i^{-1}= 2|x-y||c_i| d_i^{-1},$ and thus
$\frac{|f^i[x,y]|}{d(0,f^i[x,y])}\leq 3|x-y|
d_i^{-1}.$ Therefore
\begin{align*}
\log\frac{Df^n(x)}{Df^n(y)}&\leq\sum_{i=0}^{n-1}
\frac{|f^i[x,y]|}{d(0,f^i[x,y])}\leq \frac{3}{10}\frac{|x-y|}{D_n}\leq\frac{3}{10}.
\end{align*}
The second inequality of the lemma follows from
the fact that $e^z\leq 1+2z$ for $0\leq z\leq 3/10$.
\end{proof}

\subsection{Recovering expansion}\label{recovery}
For $p\geq1$, let
$\delta_p=\sqrt{e^{-\varepsilon p}D_{p}}$. 
 Here, $\varepsilon>0$ is the small constant in the
statement of the proposition. 
Let $\hat\delta=\delta_{10}$.
The proof of the next lemma is a slight modification of that of \cite[Lem. 1]{BC85},
 and hence it is omitted; the next lemma ensures an exponential growth of derivatives
outside of $(-\hat\delta,\hat\delta)$.
\begin{lemma}\label{exp}
If $f=f_a$, $x\in X$, $n\geq1$
are such that $|f^ix|\geq\hat\delta$ for every $0\leq i\leq n-1$, then
$|Df^{n}(x)| \geq \hat\delta e^{\lambda n }$. Moreover, if
$|f^nx|<\hat\delta$ then $|Df^{n}(x)| \geq e^{\lambda n
}$.\end{lemma} 
 To deal with the loss of expansion due
to returns to $(-\hat\delta,\hat\delta)$, we
mimic the binding argument of Benedicks $\&$ Carleson
\cite{BC85,BC91}: subdivide the interval into pieces,
and deal with them independently.
Key ingredients are the notion of binding and an associated
expansion estimate. We develop them in a slightly different way
from \cite{BC85,BC91} for our purpose.


\begin{lemma}\label{reclem1}
If $p>10$ and $\delta_p\leq|x|<\delta_{p-1}$ then:
\begin{itemize}
\item[(a)]
$|Df^{p}(x)|\geq e^{\frac{\lambda}{3}p}$;
\item[(b)] $p\leq\log
|x|^{-\frac{2}{\lambda}}$.
\end{itemize}\end{lemma}
\begin{proof}
 We have
\begin{align*} |Df^{p}(x)| &=
|Df^{p-1}(fx)|  |Df(x)|\geq
|Df^{p-1}(c_0)||x|\geq |Df^{p-1}(c_0)|  \delta_p \\
&\geq |Df^{p-1}(c_0)|^{\frac{2}{5}}\geq e^{\frac{\lambda}{3}p},
\end{align*}
where the first inequality follows from the bounded distortion in Lemma \ref{dist}.
For the last two inequalities we have used (A2) and $p>10$.
We also have \[\label{p} |x|^2\leq D_{p-1}\leq
(1/10)d_{p-2}\leq (1/10)|Df^{p-2}(c_0)|^{-1} \leq
(1/10)e^{-\lambda(p-2)}\leq e^{-\lambda p}.
\]
This yields
the upper estimate of $p$.
\end{proof}

In the following two lemmas, for $x\in X$ we consider a sequence
$$0\leq n_1(x)<n_1(x)+p_1(x)\leq n_2(x)<n_2(x)+p_2(x)\leq\cdots$$
of integers that is defined as follows:
$n_1=
\min\{n\geq 0\colon |f^nx|<\hat\delta\}$. Given $n_k$, define $p_k$, $n_{k+1}$ by 
$\delta_{p_k}\leq |f^{n_k}x|<\delta_{p_k-1}$ and
 $n_{k+1}=
\{n\geq n_k+p_k\colon
|f^nx|<\hat\delta\}$.
Lemma \ref{exp}
and Lemma
\ref{reclem1}(a) yield
\begin{equation}\label{sequence}|Df^{n_{k+1}-n_k-p_k}(f^{n_k+p_k}x)|\geq
e^{\lambda(n_{k+1}-n_k-p_k)}\ \ \text{and} \ \
|Df^{p_k}(f^{n_k}x)|\geq e^{\frac{\lambda}{3} p_k}.\end{equation}

\begin{lemma}\label{reclem2}
For all $0\leq
i<j,$
$|Df^{j-i}(c_i)|\geq e^{-\alpha\sqrt{j}}.$
\end{lemma}

\begin{proof}
Fix an integer $M$ such that
$\hat\delta e^{\alpha\sqrt{M}}\geq1$. Fix $a_0$ sufficiently near
$2$ such that $|Df(c_i)|\geq 3.5$ for every $0\leq i<M$, and the conclusion
of Lemma \ref{exp} holds for all $a\in[a_0,2]$.
We first consider the case where $|c_n|\geq\hat\delta$ for every $i\leq n\leq j-1$. If
$j\leq M$, then the choice of $a_0$ ensures $|Df^{j-i}(c_i)|\geq
(3.5)^{j-i}$, which is stronger than what is asserted. If $j> M$,
then by Lemma \ref{exp} and $\sqrt{j-i}\geq\sqrt{j}-\sqrt{i}$,
$$|Df^{j-i}(c_i)|\geq\hat\delta e^{\lambda(j-i)}\geq\hat\delta e^{\alpha\sqrt{j-i}}\geq\hat\delta e^{\alpha\sqrt{j}-\alpha\sqrt{i}}\geq\hat\delta e^{\alpha\sqrt{M}-\alpha\sqrt{i}}\geq e^{-\alpha\sqrt{i}}.$$

In the case where
$|c_n|<\hat\delta$ for some $i\leq n\leq j-1$,
consider the sequence $\{n_k,p_k\}_{k\geq1}$
for the orbit of $c_i$. 
If $n_k+p_k\leq j\leq n_{k+1}$ for some $k$, then
\eqref{sequence} yields
$|Df^{j-i}(c_i)|\geq\hat\delta e^{\frac{\lambda}{3}(j-i)}$,
which is $\geq e^{-\alpha\sqrt{i}}$ as proved in the first case.
If $n_k+1\leq j \leq n_k+p_k$ for some $k$, then
we have
\begin{align*}
|Df^{j-i}(c_i)|&=|Df^{n_k-i}(c_i)|\cdot|Df(c_{n_k})|\cdot|Df^{j-n_k-1}(c_{n_k+1})| \\
&\geq
e^{\frac{\lambda}{3}(n_k-i)}\cdot 2a|c_{n_k}|
\cdot 2^{-1} e^{\lambda(j-n_k)}
\geq e^{\frac{\lambda}{3}(j-i)- \alpha\sqrt{n_k}}\geq e^{-\alpha\sqrt{j}}.
\qedhere
\end{align*} 
\end{proof}
Let $N$ be a large integer, and set $\delta=\delta_N\ll\hat\delta$.
The next lemma on the growth of derivatives outside 
of $(-\delta,\delta)$  will be used to construct induced maps with arbitrarily small scale.
\begin{lemma}\label{exp2}
The following holds for all sufficiently large $N$:
if $x\in X$, $n\geq1$ are such that $|f^ix|\geq\delta$ for every $0\leq i\leq n-1$,
then $|Df^{n}(x)|\geq \delta
e^{\frac{\lambda}{3} n}$. Moreover, if $|f^nx|<\delta$ then
$|Df^{n}(x)|\geq e^{\frac{\lambda}{3} n}$.
\end{lemma}
\begin{proof}
For the orbit of $x$ consider the sequence $\{n_k,p_k\}_{k\geq1}$
and let $s$ be such that $n_s\leq n<n_{s+1}$.
\eqref{sequence} yields $|Df^{n_s}(x)|\geq e^{\frac{\lambda}{3}
n_s}$. If $n_s+p_s> n$, then Lemma \ref{dist} yields
$|Df^{n-n_s}(f^{n_s}x)|\geq(1/2) |Df(f^{n_s}x)||Df^{n-n_s-1}(c_0)|\geq a\delta
e^{\lambda(n-n_s-1)}\geq \delta e^{\frac{\lambda}{3}(n-n_s)}.$ If
$n_s+p_s\leq n$, then Lemma \ref{exp} yields
$|Df^{n-n_s}(f^{n_s}x)|\geq \hat\delta 
e^{\frac{\lambda}{3}(n-n_s)}.$ Hence
the first estimate of Lemma \ref{exp2} holds.

Lemma \ref{dist} and the definition \eqref{Theta}
yield
$|f^ix-c_{i-n_s-1}|\leq(1/5)|c_{i-n_s-1}|$ 
 for every $n_s\leq  i\leq n_s+p_s$, and thus
$|f^ix|\geq(4/5)|c_{i-n_s-1}|\geq(4/5)e^{-\alpha\sqrt{p_s}}\geq(4/5)e^{-\alpha
\sqrt{\frac{2}{\lambda}\log|f^{n_s}x|}}>|f^{n_s}x|\geq\delta.$
We have used (A3) for the second inequality and Lemma \ref{reclem1}(b) for the third.
The fourth inequality holds because
$|f^{n_s}x|<\hat\delta\ll1$.
Hence,
if $|f^nx|<\delta$
then $n_s+p_s\leq n$, and so the factor $\hat\delta$ above
can be dropped by
Lemma \ref{exp}.
\end{proof}


\subsection{Inducing to small scales}\label{slow}
For each $p> N$,
divide the interval $[\delta_p,\delta_{p-1})$ into $\left[ e^{3\varepsilon p}\right]$-number of subintervals of
equal length and denote them by $I_{p,j}$ $(j=1,2,\ldots,\left[ e^{3\varepsilon p}\right])$,
from the right to the left.
Let $I_{p,-j}=-I_{p,j}$, which is the mirror image of $I_{p,j}$ with respect to $0$.
\begin{lemma}\label{holder}
If $N$ is sufficiently large, then for every $I_{p,j}$
the following holds:
\begin{itemize}
\item[(a)] $|f^pI_{p,j}|\geq e^{-5\varepsilon p}$;
\item[(b)] $|I_{p,j}|\leq d(0,I_{p,j})^{1+\frac{\varepsilon}{3}}$;
\item[(c)] for all $x,y\in I_{p,j}$,
$\log\frac{Df^p(x)}{Df^p(y)}\leq |f^px-f^py|^{\varepsilon^2}.$
\end{itemize}
\end{lemma}

\begin{proof}
We have
\begin{align*}
|fI_{p,j}|
\geq e^{-3\varepsilon p}\left|\delta_{p-1}-\delta_p\right| \delta_p 
\geq e^{-4\varepsilon p}\left(e^{\frac{\varepsilon}{2}}-1\right)D_p,
\end{align*}
 and thus
\begin{align*}
|f^pI_{p,j}|\geq(1/2)|Df^{p-1}(c_0)| |fI_{p,j}|&\geq
e^{-4\varepsilon p}\left(e^{\frac{\varepsilon}{2}}-1\right) |Df^{p-1}(c_0)|D_p.\end{align*}
Using (A3) and Lemma \ref{reclem2} to estimate
the second factor we have
\begin{equation*}
|Df^{p-1}(c_0)|^{-1}D_{p}^{-1} =\sum_{j=0}^{p-1}
|c_j|^{-1}\frac{|Df^{j}(c_0)|}{|Df^{p-1}(c_0)|}\leq
pe^{2\alpha\sqrt{p}}\leq e^{3\alpha\sqrt{p}}.
\end{equation*}
Taking reciprocals and plugging the result into the above inequality,
\begin{align*}
|f^pI_{p,j}|\geq(e^{\frac{\varepsilon}{2}}-1)e^{-4\varepsilon p-3\alpha\sqrt{p}}\geq
e^{-5\varepsilon p }.\end{align*}
Hence (a) holds.

Using (A3)
we have
\begin{equation}\label{size}
\frac{\delta_{p-1}^2}{\delta_{p}^2}
\leq e^{\varepsilon}\left(1+\frac{d_{p-1}}{d_{p}}\right)=
e^{\varepsilon}\left(1+\frac{|c_{p-1}|}{|Df(c_{p-1})||c_{p}|}\right)\leq
3e^{\alpha\sqrt{p}}.\end{equation}
Hence
 $|I_p|\leq \delta_{p-1}\leq
\delta_{p}\sqrt{3e^{\alpha\sqrt{p}}},$ and thus $|I_{p,j}|\leq e^{-3\varepsilon
p}|I_p|\leq e^{-\varepsilon p}\delta_{p}$. Since $e^{-\varepsilon p}\leq
10^{-\frac{\varepsilon p }{3}}\leq\delta_p^{\frac{\varepsilon}{3}}$ we have
$|I_{p,j}|\leq e^{-\varepsilon p}\delta_{p} \leq \delta_p^{1+\frac{\varepsilon}{3}}\leq
d(0,I_{p,j})^{1+\frac{\varepsilon}{3}},$ and (b) holds.

We have $|Df^p(x)-Df^p(y)|\leq I+I\!I$, where
\begin{align*}
I=|Df^{p-1}(fx)||Df(x)-Df(y)|,\ \
I\!I=   |Df(y)||Df^{p-1}(fy)|\left|\frac
{Df^{p-1}(fx)}{Df^{p-1}(fy)}-1\right|.
\end{align*}
By (b),
\begin{equation}\label{I1}
I\leq  8|Df^{p-1}(c_0)||x-y|\leq 8|Df^{p-1}(c_0)|
d(0,I_{p,j})|x-y|^{\frac{\varepsilon}{3+\varepsilon}} .
\end{equation}
The second inequality of Lemma \ref{dist} gives
\[\left|\frac{Df^{p-1}(fx)}{Df^{p-1}(fy)}-1\right|\leq
\frac{|x-y|^2}{D_{p-1}}\leq\frac{|x-y|^2}{d(0,I_{p,j})^2}
\leq|x-y|^{\frac{2\varepsilon}{3+\varepsilon}}.\] 
Using this and 
$|Df(y)|\leq 2\cdot d(0,I_{p,j})$ which follows from (b) we get
\begin{equation}\label{I2}
I\!I\leq
4d(0,I_{p,j})|Df^{p-1}(c_0)||x-y|^{\frac{2\varepsilon}{3+\varepsilon}}.\end{equation}
Combining (\ref{I1}) (\ref{I2}) with $|Df^p(y)|\geq
|Df^{p-1}(c_0)|d(0,I_{p,j})$
yields
$$\left|\frac{Df^p(x)}{Df^p(y)}-1\right|\leq 8
|x-y|^{\frac{\varepsilon}{3+\varepsilon}}\leq |f^px-f^py|^{\frac{\varepsilon}{3+\varepsilon}}\leq|f^px-f^py|^{\varepsilon^2},$$ which implies (c). The last inequality is because
$|f^px-f^py|<1$ and $\frac{\varepsilon}{3+\varepsilon}>\varepsilon^2$.
\end{proof}

\subsection{Combinatorics of partitions}\label{combinatorics}
Let $\Lambda^+=I_{N,1}$ (the right extremal $I_{p,j}$-interval), 
$\Lambda^-=-\Lambda^+$ and $\Lambda=\Lambda^-\cup\Lambda^+.$
By induction on the number of iterations we construct a ``decreasing" sequence
$\{\tilde{\mathcal P}_n\}_{n=0}^\infty$ of partitions of
$\Lambda$ into intervals, and introduce the notion of bound/free states.
Start with $\tilde{\mathcal P}_0=\{\Lambda^+,\Lambda^-\}$. We refer to $\Lambda^\pm$ and $f^N\Lambda^\pm$
as {\it free} and  to $f^i\Lambda^\pm$ $(1\leq i\leq N-1)$ as {\it bound}.
Call $p_0(\Lambda^\pm)=N$ a {\it bound period of $\Lambda^\pm$ at time $0$}.

Set $\tilde{\mathcal P}_0=\tilde{\mathcal P}_1=\cdots=\tilde{\mathcal P}_{N-1}$,
and let $n\geq N$.
The $f^n$-images of elements of $\tilde{\mathcal
P}_{n-1}$ are in two phases: either \emph{bound} or \emph{free}.
If $\omega\in\tilde{\mathcal P}_{n-1}$, $f^n\omega$ is free and $d(0,f^n\omega)<\delta$,
then $\tilde{\mathcal P}_n$ subdivides $\omega$. For each 
resulting element 
$\omega'\in\tilde{\mathcal P}_n|\omega$ with $d(0,f^n\omega')<\delta$
an integer $p_n(\omega')$ is attached; this integer is 
called a \emph{bound period of $\omega'$ at time $n$}. We say $n$ is a {\it free return time} of $\omega'$.

Given $\omega\in\tilde{\mathcal P}_{n-1},$ $\tilde{\mathcal
P}_{n}|\omega$ is defined as follows. 
If
$f^n\omega$ is free and
contains at least two $I_{p,j}$-intervals, then let $\tilde{\mathcal P}_n$
subdivide $\omega$ according to the $(p,j)$-locations of its
$f^n$-image. 
In all other cases, let $\tilde{\mathcal
P}_{n}|\omega=\{\omega\}$.
Partition points are inserted only to ensure that
the $f^n$-images of $\tilde{\mathcal P}_n$-elements intersecting $(-\delta,\delta)$ contain
exactly one $I_{p,j}$.
$f^n$-images out of $(-\delta,\delta)$ are treated as follows.
Let $\omega'\subset\omega$ be such that $f^n\omega'$ is a component of $f^n\omega\setminus
(-\delta,\delta)$. We let $\omega'\in\tilde{\mathcal P}_n$ if $|f^n\omega'|\geq| \Lambda^+|$.
Otherwise, we glue $\omega'$ to the adjacent element whose $f^n$-image
contains $\Lambda^\pm$.

The bound periods at time $n$ of the elements of
$\tilde{\mathcal P}_n|\omega$
 are determined by the $p$-locations of their $f^n$-images. 
 Namely, if $\tilde{\mathcal P}_n$ subdivides $\omega$,  
 $\omega'\in\tilde{\mathcal P}_n|\omega$ and $p_n(\omega')$ makes sense, then
 $p_n(\omega')=p$ where $p$ is such that $f^n\omega'\supset I_{p,j}$ holds for some $j$.
 If $\tilde{\mathcal P}_n|\omega=\{\omega\}$ and $p_n(\omega)$ makes sense, then
 $p_n(\omega)=\min\{p\colon I_p\cap f^n\omega\neq\emptyset\}$.

Let $\omega'\in\tilde{\mathcal P}_n$.
We say $f^{n+1}\omega'$ is {\it bound} if
there exists $k\leq n$ such that $\omega'\in\tilde{\mathcal P}_k$,
$p_k(\omega')$ makes sense and satisfies
$n+1<k+p_k(\omega')$. Otherwise, we say $f^n\omega'$ is {\it
free}.

We need a couple of lemmas on the elements of the partitions.

\begin{lemma}\label{izo}
There exist $C_\varepsilon>1$, $C_\delta>1$ such that
if $\omega\in\tilde{\mathcal P}_{n-1}$ and $f^n\omega$ is free,
then the following holds for all $x,y\in\omega$:
\begin{itemize} 
\item[(a)] 
$\frac{Df^n(x)}{Df^n(y)}<C_\delta;$
\item[(b)] Moreover, if $f^n[x,y]\subset(-\delta,\delta)$, then 
$\frac{Df^n(x)}{Df^n(y)}\leq C_\varepsilon$.
\end{itemize}
\end{lemma}
\begin{proof}
Let $n_1<\cdots<n_s<n$ denote all the free return times in the
first $n$-iterates of $\omega$, with $p_1,\ldots,p_s$ the
corresponding bound periods defined as above. 
We decompose the time interval $[n_j+p_j,n]$ into
bound and free segments, and then apply Lemma \ref{reclem1} to
each bound segment and Lemma \ref{exp2} to each free segment.
This yields $|f^{n_j+p_j}[x,y]|\leq
\delta^{-1}e^{-\frac{\lambda}{3}(n-n_j-p_j)}|f^n[x,y]|.$
If $f^n[x,y]\subset(-\delta,\delta)$, then $\delta$ can be dropped by
 the last assertion of Lemma \ref{exp2}.

For each bound segment, using this estimate and Lemma \ref{holder}(c) we get
$$
\log\frac{Df^{p_j}(f^{n_j}x)} {Df^{p_j}(f^{n_j}y)}\leq
|f^{n_j+p_j}[x,y]|^{\varepsilon^2}\leq
\delta^{-\varepsilon^2}e^{-\varepsilon^2\frac{\lambda}{3}
(n-n_j-p_j)}\cdot|f^n[x,y]|^{\varepsilon^2}.$$
Therefore
$$\sum_{i\in\cup_{j=1}^s(n_j,n_j+p_j)}\log
\frac{Df(f^{i}x)}{Df(f^{i}y)}\leq\delta^{-\varepsilon^2}\sum_{k=1}^\infty e^{-\varepsilon^2\frac{\lambda}{3}
k}\cdot|f^n[x,y]|^{\varepsilon^2}.$$
 For free segments we have
\begin{align*}
\sum_{i\in[0,n)\setminus\cup_{j=1}^s(n_j,n_j+p_j)}\log
\frac{Df(f^{i}x)}{Df(f^{i}y)}
&\leq4\delta^{-1}\sum_{i\in[0,n)\setminus\cup_{j=1}^s(n_j,n_j+p_j)}
|f^i[x,y]| \\
&\leq4\delta^{-2}|f^n[x,y]|\sum_{i=0}^{n-1}
e^{-\frac{\lambda}{3}(n-i)}.\end{align*}
Set
$C_\delta=\exp\left({\delta^{-3}}\right)$.
Then (a) holds.
If
 $f^n[x,y]\subset (-\delta,\delta)$,  then
the multiplicative constants $\delta^{-\varepsilon^2}$ and $\delta^{-2}$
on the right-hand-sides can be replaced by $1$ and $\delta^{-1}$ respectively. 
Set $C_\varepsilon=
\exp\left(10\sum_{k=1}^\infty e^{-\varepsilon^2\frac{\lambda}{3}k}\right).$
Then (b) holds.
\end{proof}

We define inductively
a sequence $\mathcal F_1,\mathcal F_2,\ldots$ of partitions of a (full measure) subset of 
$\Lambda$ and a sequence $S_1,S_2,\ldots$ of {\it stopping time functions}
for which the following holds for every $k\geq1$:

\begin{itemize}
\item $\mathcal F_k\subset 
\bigcup_{n\geq N}\tilde{\mathcal P}_n$;

\item for each $\omega\in\mathcal F_k$, 
$f^{S_k(\omega)}\omega=\Lambda^+$ or $=\Lambda^-$, and
$f^{S_k(\omega)}$ maps a neighborhood
of $\omega$ 
diffeomorphically onto $3\Lambda^+$ or $3\Lambda^-$ (the intervals centered
at the midpoint of $\Lambda^\pm$ and three times its length). 
\end{itemize}

Start with $k=1$. Let $n\geq N$ and 
$\omega\in\tilde{\mathcal P}_{n-1}$. If $f^n\omega$ is free and
 $f^n\omega\supset3\Lambda^+$ or $\supset3\Lambda^-$, then
set $\omega'=\omega\cap f^{-n}\Lambda^{+}$ or 
$\omega'=\omega\cap f^{-n}\Lambda^{-}$, which is an element of $\tilde{\mathcal P}_n$.
Let
$\omega'\in\mathcal F_1$
and $S_1(\omega')=n$.
Subsequently we iterate the remaining parts $f^n\omega\setminus\Lambda^+$
or $f^n\omega\setminus\Lambda^-$ 
and repeat the same construction. By Lemma \ref{escape} below,
$\mathcal F_1,\mathcal F_2,\ldots$ are partitions of a full measure subset of $\Lambda$.

Given $\mathcal F_k,S_k,$ let $\omega\in\mathcal F_k$.
Without loss of generality we may assume $f^{S_k(\omega)}\omega=\Lambda^+$.
Define $\mathcal F_{k+1}|\omega$ to be the pull-back
of $\mathcal F_1|\Lambda^+$ under $f^{S_k(\omega)}|\omega$.
For $\omega'\in \mathcal F_{k+1}|\omega$ define
$S_{k+1}(\omega')=S_k(\omega)+S_1(f^{S_k(\omega)}\omega')$.

\subsection{Inducing to large scales}\label{return} 
Let $m>0$ and $\omega\in\tilde{\mathcal P}_{m-1}$ be such that $f^m\omega$ is free.
Multiple stopping times can occur in the first $m$-iterates of $\omega$.
Define
$$e(\omega)=\min\{k\geq1\colon\text{$\omega$
is not contained in an element of $\mathcal F_k$}\},$$
and
consider the conditional probability 
$$|\{S_{e(\omega)}\geq m+ n|\omega\}|=\frac{1}{|\omega|}\left|\bigcup\{\omega'\in \mathcal F_{e(\omega)}|\omega\colon
 S_{e(\omega)}(\omega')\geq m+ n\}\right|\in[0,1].$$

\begin{lemma}\label{escape}
There exist $n_0'>0$, $C>0$ and $\zeta\in (0,1)$  
such that 
if $m\geq  n_0'$, $\omega\in\tilde{\mathcal P}_{m-1}$
and $f^m\omega$ is free, then 
$|\{S_{e(\omega)}\geq m+ n|\omega\}|\leq C\zeta^n$ for every $n\geq \varepsilon^{1/2} m.$
\end{lemma}
\begin{proof} 
Let $\mathcal G=\{\omega'\in \mathcal F_{e(\omega)}|\omega\colon
 S_{e(\omega)}(\omega')\geq m+ n\}$.
Let 
$\mathcal G'$ denote the set of all 
$\omega'\in\mathcal G$ for which there exists $m\leq k<m+n$ such that 
$d(0,f^k\omega_k')
<\delta$ holds for the element $\omega_k'\in\tilde{\mathcal P}_k$ containing $\omega'$. Let 
$\mathcal G''=\mathcal G\setminus\mathcal G'$.

 Each $\omega'\in\mathcal G'$ 
has an
\emph{itinerary} $(n_1,p_1,j_1),
\ldots,(n_s,p_s,j_s)$ that is defined as follows:
$m\leq n_1<\cdots<n_s< m+n$ is a sequence of integers,
associated with a nested sequence $\omega\supset\omega_{n_1}\supset\cdots\supset
\omega_{n_s}\supset\omega'$ of intervals such that for each $i$, $\omega_{n_i}$ is
the element of $\tilde{\mathcal P}_{n_i}$ containing $\omega'$ that arises out of the subdivision at time $n_i$, with $d(0,f^{n_i} \omega_{n_i})<\delta$ and $(p_i,j_i)$ its
$(p,j)$-location. 
Let $n_{s+1}\geq m+n$ be such that $\tilde{\mathcal
P}_{n_{s+1}}$ partitions $\omega_{n_s}$.
For any $x\in\omega'$ we have
$|Df^{n_{s+1}}(x)|\geq \delta e^{\frac{\lambda}{3}\sum_{i=1}^sp_i}|Df^m(x)|$
and
$|Df^m(x)|\geq C_\delta^{-1}|f^m\omega|/|\omega|$, and thus
$|\omega'|\leq|\omega_{n_s}|\leq C_\delta\delta^{-1}e^{-\frac{\lambda}{3}\sum_{i=1}^sp_i}|\omega|/|f^m\omega|.$
Then
\begin{align*}\label{ret}\sum_{\omega'\in\mathcal G'}|\omega'|&=
\sum_s\sum_{P} \sum_{\stackrel{\{(n_i,p_i,j_i)\}_{i=1}^s}{\sum_{i=1}^sp_i=P}}|
\omega'|\\&\leq C_\delta\delta^{-1}\frac{|\omega|}{|f^m\omega|}\sum_s\sum_{P}e^{-\frac{\lambda}{3}P}
\#\left\{\{(n_i,p_i,j_i)\}_{i=1}^s\colon \sum_{i=1}^sp_i=P\right\}.\end{align*}
The integer $s$ in the summand ranges up to $[n/N]$.
Since $1\leq |j_i|\leq 
e^{3\varepsilon p_j}$,
the number of all sequences $\{(p_i,j_i)\}_{i=1}^s$
with $\sum_{i=1}^sp_i=P$ is
$\leq 2^s\left(\begin{smallmatrix}
P+s\\s\end{smallmatrix}\right)e^{3\varepsilon P}.$
Since there are at most $\left(\begin{smallmatrix}
n\\s\end{smallmatrix}\right)$ number of ways of distributing
$n_1,\ldots,n_s$ in $[m,m+n)$, by the Stirling formula for factorials we get
\begin{equation}\label{counta}
\#\left\{\{(n_i,p_i,j_i)\}_{i=1}^s\colon \sum_{i=1}^sp_i=P\right\}\leq  \begin{pmatrix}
n\\s\end{pmatrix}2^s\begin{pmatrix}
P+s\\s\end{pmatrix}e^{3\varepsilon P}\leq e^{\varepsilon n}e^{4\varepsilon P}.\end{equation}

\begin{sublemma}\label{subl}
For every $1\leq
i\leq s$,
$n_{i+1}-n_i\leq 2p_i.$
\end{sublemma}
\begin{proof}
Lemma \ref{holder}(a) gives $|f^{n_i+p_i}\omega_i|\geq e^{-5\varepsilon p_i}$,
and thus $n_{i+1}-n_i-p_i\leq\frac{6\varepsilon p_i}{\lambda}-\log \delta$;
for otherwise, Lemma \ref{exp} would yield $|f^{n_{i+1}}\omega_i|>2$, which is 
a contradiction. A simple computation shows $-\log\delta<N\leq p_i$, and
thus $n_{i+1}-n_i\leq 2p_i$.
\end{proof}
It follows that 
$n<n_{s+1}\leq n_1+2\sum_{i=1}^{s}p_i.$
If $n_1\leq m+n/2$ then
$\sum_{i=1}^{s}p_i\geq n/4$, and therefore 
\begin{equation}\label{ret1}\sum_{\stackrel{\omega'\in\mathcal G'}{n_1\leq m+n/2}}|\omega'|
\leq C_\delta\delta^{-1}\frac{n}{N}\sum_{P\geq n/4}
e^{(8\varepsilon-\frac{\lambda}{3})P}\frac{|\omega|}{|f^m\omega|}\leq
e^{-\frac{\lambda}{13}n}\frac{|\omega|}{|f^m\omega|},\end{equation}
where the last inequality holds provided $m$ is sufficiently large
because $n\geq \varepsilon^{1/2} m$.

For those $\omega'\in\mathcal G'$ with $n_1>m+n/2$, 
a similar reasoning shows
$$|\omega'|\leq|\omega_{n_1}|\leq C_\delta\delta^{-1} e^{-\lambda(n_1-m)}
\frac{|\omega|}{|f^m\omega|}\leq  C_\delta\delta^{-1} e^{-\frac{\lambda n}{2}}
\frac{|\omega|}{|f^m\omega|}
\leq  e^{-\frac{\lambda n}{3}}
\frac{|\omega|}{|f^m\omega|},$$
and therefore
\begin{equation}\label{ret2}\sum_{\stackrel{\omega'\in\mathcal G'}{n_1> m+n/2}}|\omega'|
\leq C_\delta\delta^{-1}
\frac{n}{N}\sum_{P\leq n}
e^{4\varepsilon P+\varepsilon n-\frac{\lambda}{3}n}\frac{|\omega|}{|f^m\omega|}\leq
e^{-\frac{\lambda n}{4}}\frac{|\omega|}{|f^m\omega|},\end{equation} 
where the last inequality holds provided $m$ is sufficiently large 
because $n\geq \varepsilon^{1/2}m$.
\medskip

 We now treat elements of $\mathcal G''$.
 Let $r\geq 0$ denote the integer such that $\omega$ is subdivided at time $m+r$.
 Since $I_{pj}\supset f^k\omega$ holds for some $k<m$
 we have $|f^m\omega|\geq\delta e^{-5\varepsilon m}$.
If $r\geq \varepsilon^{1/2} m$, then 
$|f^{m+n}\omega|\geq
\delta e^{-5\varepsilon m}e^{\lambda\varepsilon^{1/2} m}>2=|X|,$
which is a contradiction. 
 Hence $r<\varepsilon^{1/2} m$, and thus $r<n$. 
 
Let $k\geq m$. Let us say that
$\tilde\omega\in\tilde{\mathcal P}_k|\omega$ is an {\it escaping component at time $k$}
if $\tilde\omega$ arises out of subdivision at time $k$ and satisfies
$d(0,f^k\tilde\omega)=\delta$.
 Let $\mathcal E_1$ denote the collection of escaping components at time $r$.
 If $\mathcal E_1=\emptyset$, then $\mathcal G''=\emptyset$.
 Hence we assume $\mathcal E_1\neq\emptyset$. 

Each $\omega'\in\mathcal G''$ 
has an
\emph{itinerary} $(k_1,\epsilon_1),
\ldots,(k_t,\epsilon_t)$ that is defined as follows:
$m\leq k_1<\cdots<k_t< m+n$ is a sequence of integers,
associated with a nested sequence $\omega\supset\omega_{k_1}\supset\cdots\supset
\omega_{k_t}\supset\omega'$ of intervals such that for each $i$, $\omega_{k_i}$ is
an escaping component at time $k_i$ and $\epsilon_i=+$ (resp. $\varepsilon_i=-$)
if $f^{k_i} \omega_{k_i}$ is at the right (resp. left) of the critical point.
Call $t$ the {\it length of the itinerary of} $\omega'$.
Using the previous estimates and the fact that $\omega_{k_t}$ is not subdivided up to time
$m+n-1$, we have 
$|\omega_{k_t}|\leq e^{-\lambda n/2}|\omega|/|f^m\omega|.$

Let $\mathcal H=\{\omega'\in\mathcal G''\colon\text{The length of the itinerary is $\leq \theta
n$}\}$.
The number of all itineraries of length $t$ is $\leq\left(\begin{smallmatrix}n\\t\end{smallmatrix}\right)$, and so by the Stirling formula one can choose
a small constant $\theta>0$ such that $\#\mathcal H\leq e^{\lambda n/100}$.
Then
\begin{equation}\label{ret3}\sum_{\omega'\in
\mathcal H}|\omega'|\leq\#\mathcal H e^{-\lambda n/2}\frac{|\omega|}{
|f^m\omega| }\leq
e^{-\frac{\lambda}{3}n}\frac{|\omega|}{|f^m\omega|}.\end{equation}

To treat elements in 
$\mathcal H'=\{\omega'\in\mathcal G''\colon\text{The length of the itinerary is $\geq \theta
n$}\}$, for each $t\geq1$
 define a collection $\mathcal E_t$ of escaping components (at variable times) inductively as follows: 
  each $\omega\in\mathcal E_t$ is an escaping component at 
some time, say $k=k(\omega)$.
Let $k'>k$ denote the time at which $\omega$ is subdivided. 
Then $\omega$ contains no or at most two escaping components at time $k'$. 
We let them in $\mathcal E_{t+1}$.
Let $E_t=\bigcup_{\omega\in\mathcal E_t}\omega$.
The bounded distortion in Lemma \ref{izo} implies that 
there exists $\hat\zeta\in(0,1)$ 
such that
for every $t\geq1$ and $\omega\in\mathcal E_t$,
 $|\omega\cap\Omega_{t+1}|\leq(1-\hat\zeta)|\omega|.$ Hence
 $|E_{t+1}|\leq(1-\hat\zeta)|E_t|,$ and thus
$|E_t|\leq(1-\hat\zeta)^{t}|\omega|$.
By definition, if the itinerary of $\omega'\in\mathcal H'$ is of length $t$,
then $\omega'$ is contained in an element of $\mathcal E_t$.
Hence
 \begin{equation}\label{ret4}\sum_{\omega'\in
\mathcal H'}|\omega'|\leq\sum_{\theta n\leq t\leq n}|E_t|\leq
\sum_{t\geq\theta n}(1-\hat\zeta)^t\leq \hat\zeta^{-1}(1-\hat\zeta)^{ \theta n}.\end{equation}
Set $C=1+\hat\zeta^{-1}$ and $\zeta=\max\{e^{-\frac{\lambda}{14}},(1-\hat\zeta)^{ \theta }\}$.
\eqref{ret1} \eqref{ret2} \eqref{ret3} \eqref{ret4} yield 
 $ |\{S_{e(\omega)}\geq m+ n|\omega\}|
\leq e^{-\frac{\lambda}{14}n}+\hat\zeta^{-1}(1-\hat\zeta)^{ \theta n}\leq C\zeta^n.$
 \end{proof}

\section{Proof of the proposition}
In this last section we prove the proposition. 
In Sect. \ref{positive} we construct a Cantor set $\Omega_\infty$ of positive Lebesgue measure. In Sect. \ref{induced} and Sect. \ref{tower} we construct an induced map $F\colon \Omega_\infty\circlearrowleft$ and then define an associated tower $\Delta$. In Sect. \ref{recurrence} we show that this tower has a distinctive property, and in Sect. \ref{horseshoe} use this property to construct a certain convenient horseshoe. In Sect. \ref{meas} we construct an invariant probability measure with the properties in the statement of the proposition.

\subsection{Construction of a positive measure set}\label{positive}
We construct a subset $\Omega_\infty$ of $\Lambda$ with positive
Lebesgue measure. Let $\Omega_{N-1}=\Lambda$.
For $n\geq N$
we inductively define
$$\Omega_n=\Omega_{n-1}\setminus \bigcup
\{\omega\in\tilde{\mathcal P}_n\colon d(0,f^{n}\omega)<\delta
e^{-\varepsilon n}\},$$
 and set
$\Omega_\infty=\bigcap_{n\geq N-1}\Omega_{n}$. Any component of
$\Omega_{n-1}\setminus\Omega_n$ is called a {\it gap of order
$n$}.

\begin{lemma}\label{positi}
For any $n\geq0$ and $\tilde\omega\in\tilde{\mathcal P}_n$ we have
$|\tilde\omega\cap\Omega_\infty|\geq(1/2)|\tilde\omega|.$
In particular,
$|\Omega_\infty|\geq (1/2)|\Lambda|.$
\end{lemma}

\begin{proof}
Choose a point $x\in \tilde\omega\cap\Omega_\infty$.
Since
$|f^ix|\geq\delta e^{-\varepsilon i}$ for every
$N\leq i\leq n$, the element of $\tilde{\mathcal P}_n$ containing $x$, which is $\tilde\omega$, 
belongs to $\Omega_n$.
Let $k\geq n+1$ and $\omega\in\tilde{\mathcal P}_{k-1}|\tilde\omega$ be such that
$\omega\subset\Omega_{k-1}$, and suppose that some part of it is deleted at step $k$. 
We claim that $f^k\omega$ is free. Indeed, if this is false then for the last free return
time $j$ of $\omega$ before $k$ with bound period $p$ we have $k<j+p$. We also have
$|f^{j+1}\omega|\leq 2\delta_{p}^2,$ and thus $|f^{k}\omega|\leq
|Df^{k-j-1}(c_0)|D_{p}\leq(1/10)|c_{k-j}|.$ This yields
\begin{equation*}\label{inessential2}
d(0,f^k\omega)\geq(9/10)|c_{k-j}|
\geq(9/10)e^{-\alpha\sqrt{p}}
\geq(9/10)e^{-\alpha\sqrt{\frac{2}{\lambda}(-\log\delta+\varepsilon k)}}>
(\delta/2)e^{-\varepsilon k},\end{equation*}
which means that no part of $\omega$ is deleted at step $k$ and a contradiction arises.

\begin{sublemma}
$|f^k\omega\cap(-\delta,\delta)|\geq\delta e^{-\frac{10\varepsilon^2}{\lambda}k}.$
\end{sublemma}
\begin{proof}
 If $f^k\omega\subset(-\delta,\delta)$ then let $i<k$ be such that
$f^i\omega$ is free and contains some $I_{p,j}$. Since 
$p\leq\frac{2}{\lambda}(-\log\delta+\varepsilon k)$, Lemma \ref{holder}(a) yields
$|f^{i+p}\omega|\geq e^{-5\varepsilon p}\geq \delta^{\frac{10\varepsilon}{\lambda}}
e^{-\frac{10\varepsilon^2}{\lambda}k}$.
Since $f^k\omega$ is free we obtain 
$|f^k\omega|\geq|f^{i+p}\omega|\geq \delta
e^{-\frac{10\varepsilon^2}{\lambda}k}$.
If $f^k\omega$ is not contained in $(-\delta,\delta)$, 
then obviously
$|f^k\omega\cap(-\delta,\delta)|\geq \delta(1-e^{-\varepsilon k})>\delta e^{-\frac{10\varepsilon^2}{\lambda} k}.$
\end{proof}
The subinterval of
$f^k\omega$ to be deleted has length $\leq 3\delta e^{-\varepsilon k}.$
Taking distortions into consideration when pulling back to
$\omega$, we have
$$\frac{|\tilde\omega\cap(\Omega_{k-1}\setminus\Omega_k)|}{|\tilde\omega\cap\Omega_{k-1}|}
\leq 3C_\varepsilon e^{-\left(\varepsilon -\frac{10\varepsilon^2}{\lambda}\right)k}.$$
This yields
\[\frac{|\tilde\omega\cap\Omega_\infty|}{|\tilde\omega|}\geq\prod_{k=N}^\infty\left(1-
3C_\varepsilon e^{-\left(\varepsilon -
\frac{10\varepsilon^2}{\lambda}\right)k}\right)\geq\frac{1}{2}.\hfill{}\qedhere\]
\end{proof}

\subsection{Construction of an induced map on the Cantor set}
\label{induced}
Lebesgue almost every $x\in\Omega_\infty$ has an infinite number of
stopping times, which we denote by $S_1(x)<S_2(x)<\cdots$ with a slight abuse of
notation.
We define its subsequence
$R_1(x)<R_2(x)<\cdots$ and a return
time $R(x)$ to $\Omega_\infty$ as follows. Start with
$R_1(x)=S_1(x)$. Given $R_i(x)$, 
if $f^{R_i(x)}x\in\Omega_\infty$ then 
let $R(x)=R_i(x)$. If
$f^{R_i(x)}x\notin\Omega_\infty$, then let $g_i$ denote the order of the gap containing
$f^{R_i(x)}x$ and define $R_{i+1}(x)$ to be the smallest stopping time after $R_i(x)+g_i$.
Note that $R$ is not necessarily the first return time to $\Omega_\infty$.
The $R_i(x)$ $(i=1,2,\ldots)$ are called {\it regular return times of $x\in\Omega_\infty$}.
Let $\Omega_\infty^\pm=\Lambda^\pm\cap\Omega_\infty$.


\begin{lemma}\label{lattice}
There exists a countable partition $\mathcal Q$ of a full measure
subset of $\Omega_\infty$ such that the following holds for every
$\omega\in\mathcal Q$:
\begin{itemize}

\item[(a)] $R$
is constant on $\omega$ (denote this value by $R(\omega)$) and $f^{R(\omega)}$
maps $\omega$ bijectively onto $\Omega_\infty^\pm$. In addition,

\item[(b)] 
for all $x,y\in\omega$,
$\left|\frac{Df^{R(\omega)}(x)}{Df^{R(\omega)}(y)}-1\right|\leq 3|\Lambda^+|^{-1}
|f^{R(\omega)}[x,y]|$.


\end{itemize}
\end{lemma}

\begin{proof}
We construct $\mathcal Q$ by induction using the partitions
$\mathcal F_k$
and the stopping times $S_k$. To begin with,
for all $\omega\in \mathcal F_1$ intersecting $\Omega_\infty$ let
$(f^{S_1(\omega)}|\omega)^{-1}\Omega_\infty\in\mathcal Q.$ 
\begin{sublemma}\label{match}
If $\omega\in\mathcal F_1$ and $\omega\cap\Omega_\infty\neq\emptyset$, then
$(f^{S_1(\omega)}|\omega)^{-1}\Omega_\infty\subset
\Omega_\infty$.
\end{sublemma}
\begin{proof}
Let $x\in
(f^{S_1(\omega)}|\omega)^{-1}\Omega_\infty$.
Since $\omega\cap \Omega_\infty\neq\emptyset$,
$\omega\subset\Omega_{S_1(\omega)-1}$ holds, and thus
$x\in\Omega_{S_1(\omega)-1}$. Since
$f^{S_1(\omega)}x\in\Omega_\infty$, $x\in\Omega_{S_1(\omega)}$ holds.
Since  $f^{S_1(\omega)}x=y$ for some $y\in\Omega_\infty$,
for every $n>0$ we have
$|f^{n+S_1(\omega)}x|=|f^ny|\geq \delta e^{-\varepsilon n}>\delta
e^{-\varepsilon (n+S_1(\omega))}$, and thus $x\in\Omega_{n+S_1(\omega)}$.
This yields $x\in\Omega_\infty$. \end{proof}

For $\omega\in\mathcal F_1$ intersecting
$\Omega_\infty$, let $G$ be a gap of order $g$ with $G\cap
f^{S_1(\omega)}(\omega\cap\Omega_\infty)\neq\emptyset$. 
For any $k>1$ and
$\omega'\in\mathcal F_{k}|\omega$ such that: (i)
$\omega'\cap\Omega_\infty\neq\emptyset$;
(ii) $f^{S_1(\omega)}\omega'\subset G$; (iii) 
$S_k(\omega')=R_2(x)$ for $x\in\omega'\cap\Omega_\infty$,
let
$(f^{S_k(\omega')}|\omega')^{-1}\Omega_\infty\in\mathcal
Q$.

\begin{sublemma}\label{match2}
$(f^{S_k(\omega)}|\omega')^{-1}\Omega_\infty\subset\Omega_\infty$.
\end{sublemma}
\begin{proof}
Let $x\in
(f^{S_k(\omega)}|\omega')^{-1}\Omega_\infty.$
Since $\omega'\cap\Omega_\infty\neq\emptyset$, $x\in\Omega_{S_k(\omega)-1}$
holds. Since $f^{S_k(\omega)}x\in\Omega_\infty$, $x\in\Omega_{S_k(\omega)}$ holds.  
Reasoning as in the proof of Sublemma \ref{match}, for
every $n>0$ we have $|f^{n+S_k(\omega)}x|\geq \delta e^{-\varepsilon n}>\delta
e^{-\varepsilon(n+S_k(\omega))}$, and thus $x\in\Omega_{n+S_k(\omega)}$. This yields
$x\in\Omega_\infty$.
\end{proof}
In subsequent steps we treat points sent into gaps in the previous
steps. This completes the construction of $\mathcal Q$. 
(a) is a direct consequence of the construction.
Since $f^{R(\omega)}=3\Lambda^\pm$ and $f^{R(\omega)}|\omega$ is extended to a diffeomorphism
onto $3\Lambda^\pm$,
the Koebe Principle \cite[Chap. IV.1]{MelStr93} yields
(b).
\end{proof}
Define an induced map $F\colon\Omega_\infty\circlearrowleft$ by
$F|\omega=f^{R(\omega)}$ for $\omega\in\mathcal Q$. 
By Lemma
\ref{lattice} and \cite[Chap. V.2 Thm. 2.2]{MelStr93}, there exists
an $F$-invariant probability measure $\nu_0$ that is absolutely
continuous with respect to the Lebesgue measure ${\rm Leb}_{\Omega_\infty}$ on
$\Omega_\infty$, with the density
$d\nu_0/d{\rm Leb}_{\Omega_\infty}$ uniformly bounded away from
zero and infinity. 
The next measure estimate of the tail 
$\{R> n\} =\{x\in\Omega_\infty\colon R(x)> n\}$
implies that $\nu_0$ projects down to the acip $\mu$ for $f$.
\begin{lemma}\label{exptail}
For all large $n$, $|\{R> n\}| \leq e^{-\frac{\lambda}{10}n}$.
\end{lemma}
\begin{proof}
For $0\leq k\leq n$, let
$$\mathcal P_k'=\{\omega\in\tilde{\mathcal P}_k\colon \omega\cap\{R> n\}\neq\emptyset\}.$$
Let $\omega\in\mathcal P_k'$.
By construction, all points in $\omega$ share the same sequence of regular return times up to time $k$, which we simply denote by $0=:R_0<R_1<R_2<\cdots\leq k.$
For $i\geq0$ let $\omega_i$
denote the element of $\tilde{\mathcal P}_{R_i}$ which contains
$\omega$.

\begin{sublemma}\label{d}
$R_{i+1}-R_{i}\geq N$ and $|\omega_{i+1}|/|\omega_i|\leq
3e^{-\frac{\lambda}{3} (R_{i+1}-R_{i})}.$
\end{sublemma}
\begin{proof}
The first inequality follows from
$p(f^{R_i}\omega_i)\geq N$ and that
$f^{R_{i+1}}\omega_i$ is free. 
The mean value theorem gives 
$|\Lambda^+|=|f^{R_i}\omega_{i}|=|Df^{R_i}(x)||\omega_{i}|$ 
for some $x\in\omega_i$,
and $|f^{R_i}\omega_{i+1}|=|Df^{R_i}(y)||\omega_{i+1}|$ 
for some $y\in\omega_{i+1}$.
The Koebe Principle implies $|Df^{R_i}(x)|\leq 3|Df^{R_i}(y)|,$ and thus
$|\omega_{i+1}|/|\omega_i|\leq
3|f^{R_{i}}\omega_{i+1}|/|\Lambda^+|.$
In addition
$
|\Lambda^+|=|f^{R_{i+1}-R_i}(f^{R_{i}}\omega_{i+1})|\geq  e^{\frac{\lambda}{3}(R_{i+1}-R_i)}|f^{R_{i}}\omega_{i+1}|,$ and thus
the second inequality holds.
\end{proof}

Let $1\leq j\leq [n/N]$.
For a $j$-string $(k_1,\ldots,k_j)$ of positive
integers, let
$$\mathcal Q(k_1,\ldots,k_j)=\left\{\omega\in\mathcal P_{\sum_{i=1}^jk_i}'
\colon  R_{i}(\omega)=k_1+k_2+\cdots+ k_{i}\ \text{for every}\
1\leq  i\leq j\right\},$$
and let $|\mathcal Q(k_1,\ldots,k_j)|=\sum_{\omega\in \mathcal Q(k_1,\ldots,k_j)}|\omega|$.

\begin{sublemma}\label{sb1l}
$|\mathcal Q(k_1,\ldots,k_j)| \leq e^{-\frac{\lambda}{4}\sum_{i=1}^jk_i}.$
\end{sublemma}
\begin{proof}
For each $1\leq i<j$ and
$\omega_{i}\in \mathcal Q(k_1,\ldots,k_{i})$, let $$\mathcal
Q(\omega_{i},k_{i+1})= \{\omega_{i+1}\in\mathcal
Q(k_1,\ldots,k_{i+1})\colon \omega_{i+1}\subset\omega_{i}\}.$$ 
Using the second inequality of Sublemma \ref{d} and 
$\#\mathcal Q(\omega_{i},k_{i+1})\leq e^{6\varepsilon k_{i+1}}$ which 
follows from the proof of Lemma \ref{escape} we get
\begin{align*}
|\mathcal Q(k_1,\ldots,k_{i+1})|
&=\sum_{\omega_{i}\in \mathcal
Q(k_1,\ldots,k_{i})} |\omega_{i}|\sum_{\omega_{i+1}\in\mathcal
Q(\omega_{i},k_{i+1}) } \frac{|\omega_{i+1}|}
{|\omega_{i}|} \\
&\leq  \sum_{\omega_{i}\in \mathcal
Q(k_1,\ldots,k_{i})} |\omega_{i}|e^{6\varepsilon k_{i+1}}\cdot 3
e^{-\frac{\lambda}{3} k_{i+1}} \\
&\leq e^{-\frac{\lambda}{4} k_{i+1}} \sum_{\omega_{i}\in \mathcal
Q(k_1,\ldots,k_{i})} |\omega_{i}|=e^{-\frac{\lambda}{4}k_{i+1}}|\mathcal
Q(k_1,\ldots,k_{i})|.\end{align*}
 Using this inductively, combining the result with
$|\mathcal Q(k_1)|\leq e^{-\frac{\lambda}{4}k_1}$ and then substituting
$i=j-1$ we obtain the desired inequality.
\end{proof}

\begin{sublemma}\label{sbl2}
If $\sum_{i=1}^jk_i<[n/2]$, then
$|\mathcal Q(k_1,\ldots,k_j)|\leq e^{-\frac{\lambda}{5}n}.$
\end{sublemma}
\begin{proof}
Let $\omega\in\mathcal Q(k_1,\ldots,k_j)$.
If $f^n\omega$ is bound, then let $k<n$ denote the free return with 
bound period $p$
with $k<n<k+p$. Since $\omega$ intersects $\Omega_\infty$,
$d(0,f^k\omega)\geq \delta e^{-\varepsilon k},$
and thus $k+p-n\leq \frac{3\varepsilon}{\lambda} n$. Then for all $x\in\omega$,
$|Df^n(x)|=|Df^{k+p}(x)|/|Df^{k+p-n}(f^{n}x)|\geq 4^{-(k+p-n)}e^{\frac{\lambda}{3}(k+p)}\geq 4^{-\frac{3\varepsilon}{\lambda} n}e^{\frac{\lambda}{3}n}\geq e^{\frac{\lambda}{4}n},$
and thus $|\omega|\leq e^{-\frac{\lambda}{4}n}$.
If $f^n\omega$ is free, then 
$|\omega|\leq e^{-\frac{\lambda}{4}n}$.
Hence
$$
|\mathcal Q(k_1,\ldots,k_j)|\leq e^{-\frac{\lambda}{4}n}\#\mathcal Q(k_1,\ldots,k_j).
$$
From the proof of Lemma \ref{escape} and the assumption $\sum_{i=1}^jk_i<[n/2]$
we have
$\#\mathcal Q(k_1,\ldots,k_j)\leq e^{5\varepsilon\sum_{i=1}^jk_i}
\leq e^{\frac{5\varepsilon}{2} n}$, and so the desired inequality follows.
\end{proof}
Returning to the proof of Lemma \ref{exptail},
let 
$\mathcal Q_{n}^j$ denote the collection of elements of
$\mathcal P_n'$ for which the number of regular
return times up to time $n$ is equal to $j$.
 Let $|\mathcal Q_n^j|=\sum_{\omega\in \mathcal Q_n^j}|\omega|$.
Sublemma \ref{sb1l} and Sublemma \ref{sbl2}
yield
\begin{align*}
|\mathcal Q_{n}^j|&=
\sum_{K=1}^{[n/2]-1}\sum_{k_1+\cdots+k_j=K}|\mathcal
Q(k_1,\ldots,k_j)|+
\sum_{K=[n/2]}^n\sum_{k_1+\cdots+k_j=K}|\mathcal
Q(k_1,\ldots,k_j)|\\
&\leq e^{-\frac{\lambda}{5}n}\sum_{K=1}^{[n/2]-1} \#\left\{ (k_1,\ldots,k_j)\colon \sum_{i=1}^j
k_i=K\right\} \\
&+ \sum_{K=[n/2]}^n e^{-
\frac{\lambda}{4}K}\#\left\{ (k_1,\ldots,k_j)\colon \sum_{i=1}^j
k_i=K\right\}\\
&\leq e^{-
\frac{\lambda}{5}n}\sum_{K=1}^{[n/2]-1}e^{\beta K}+ \sum_{K=[n/2]}^n e^{-
(\frac{\lambda}{4}-\beta)K},
\end{align*}
where $\beta\to0$ as $N\to\infty$.
Hence $|\mathcal Q_{n}^j|\leq e^{-\frac{\lambda}{9}n}$ and
$|\{R > n\}|\leq\sum_{j=1}^{[n/N]}|\mathcal Q_{n}^j|\leq e^{-\frac{\lambda}{10}n}$.
\end{proof}

\subsection{Reduction to lower floors of the tower}\label{tower}
Let
$$\Delta=\{(x,l)\colon x\in \Omega_\infty,\ \ l=0,1,\ldots,R(x)-1\},$$
which we call a {\it tower}, and define
$$\hat f(x,l)=\begin{cases}(x,l+1)\ \ \text{ if }l+1<R(x) &\\
(f^{R(x)} x,0)\ \text{ if }l+1=R(x).\end{cases}$$ The point
$(x,l)$ is considered to be climbing the tower in the first case and
falling down from the tower in the second case. Define a
projection $\pi\colon\Delta\to X$ by $\pi(x,l)=f^lx$.
 Let
 $\Delta_l=\{(x,l)\in\Delta\colon R(x)>l\}.$ Note that
$\Delta_0=\{(x,0)\colon x\in \Omega_\infty\}$. Let
$\tau_l\colon\{R>l\}\to\Delta_l$ denote the canonical
identification $\tau_l(x)=(x,l)$.
Fix a measurable structure on $\Delta$ such that $\pi$ is
measurable, and define a probability
measure $\hat \mu$ on $\Delta$ by $$\hat \mu=
\frac{1}{\nu_0(R)}\sum_{l=0}^\infty (\tau_l)_*\nu_0|\{R>l\},$$
where  $\nu_0(R)<\infty$ by Lemma \ref{exptail}.
Observe that $\pi_*\hat \mu=\mu$.

We reduce the desired upper estimate
in the proposition to an estimate on the lower floors of the tower.
For each $l\geq 0$, let $\mathcal P_l=\tilde{\mathcal P}_l|\{R>l\}$. 
Using $\tau_l$ we transplant the partition
$\mathcal P_l$ to $\Delta_l$ and also denote it by $\mathcal P_l$. Let 
 $\mathcal D=\bigcup_{l\geq0}\mathcal P_l$ denote
the resultant partition of $\Delta$. Let $\hat\varphi_j=
\varphi_j\circ\pi$. Let
$$\mathcal B_n=\left\{A\in\bigvee_{i=0}^{n-1} \hat f^{-i}\mathcal
D\colon \frac{1}{n}S_n\hat\varphi_j(x)\geq b_j\ \ j=1,\ldots,d
\text{ for some $x\in A$} \right\}.$$ 
 If $\mathcal
B_n=\emptyset$ there is nothing to prove, and hence we assume
$\mathcal B_n\neq\emptyset$.
We have
$$\frac{1}{n}\log \mu\left\{\frac{1}{n}S_n\varphi_j \geq b_j
\right\}=\frac{1}{n}\log\hat \mu\left\{\frac{1}{n}S_n\hat\varphi_j
\geq b_j\right\}\leq\frac{1}{n}\log\hat\mu(\mathcal B_n),$$ where $\hat\mu(\mathcal
B_n)=\sum_{A\in\mathcal B_n}\hat \mu(A)$. 
Define
$$\mathcal B_n'=\left\{A\in\mathcal B_n\colon A\subset
\bigcup_{0\leq l\leq 30n}\Delta_l\right\}\ \ \text{and} \ \ \mathcal
B_n''=\left\{A\in\mathcal B_n\colon
A\subset\bigcup_{l>30n}\Delta_l\right\}.$$ 
Let $p_1\colon\Delta\to\Lambda$ denote the projection to the first coordinate. 
The next lemma enables us to compare $\hat\mu$ and ${\rm Leb}_{\Omega_\infty}$.
\begin{lemma}\label{mhat}
There exist $0<C_1<C_2$ such that for any $l\geq0$ and any 
measurable set $A\subset\Delta_l$ we have
$C_1|p_1A|\leq\hat \mu(A)\leq C_2|p_1A|$.
\end{lemma}
\begin{proof}
For any measurable $A\subset\Delta_l$ we have
 $\hat\mu(A)=\nu_0(p_1A)/ \nu_0(R)$, and 
the density of $\nu_0$ is uniformly bounded away from zero and infinity.
Hence the claim holds.
\end{proof}

By Lemma
\ref{exptail} and Lemma \ref{mhat},
\begin{equation*}\hat\mu(\mathcal B_n'')\leq
C_2\sum_{l>30n} |\{R>l\}| < 4
^{-n},\end{equation*} 
and thus for any $\nu\in\mathcal M_f$,
\begin{equation*}
\varlimsup_{n\to\infty}\frac{1}{n}\log\hat\mu(\mathcal B_n'')\leq-\log4\leq F(\nu).\end{equation*}
For the proof of the proposition
it suffices to show that for all large $n$ such that  $\mathcal B_n'\neq\emptyset$
there exists $\sigma\in\mathcal M_f$ satisfying \eqref{upper5} such that
\begin{equation}\label{upper3}
\frac{1}{n}\log\hat\mu(\mathcal B_n')\leq(1-\varepsilon^{1/5})F(\sigma)+2\varepsilon^{1/5}.
\end{equation}
For the rest of this paper we
assume $\mathcal B_n'\neq\emptyset$.

\subsection{Approximation by points quickly falling down from the tower}\label{recurrence}
For $n\geq N$ and
$\omega\in\mathcal P_n$, let $\tilde\omega$ denote the element of
$\tilde{\mathcal P}_n$ containing $\omega$, namely
$\omega=\tilde\omega\cap\{R>n\}$. 
Points in $\omega$ may climb the tower for a very long
period of time.
The next lemma indicates that
a positive definite fraction of points in $\tilde\omega$
quickly fall down to the ground floor $\Delta_0$.

\begin{lemma}\label{Mar}
There exists $n_0''\geq n_0'$ such that
for every $n\geq n_0''$  and
$\omega\in\mathcal P_n$ there exist  
 $\omega'\subset\tilde\omega\cap\Omega_\infty$ and
 $r=r(\omega')\in[n,(1+\varepsilon^{1/3})n]$ 
such that:

\begin{itemize}

\item[(a)] $|\omega'|\geq
e^{-\varepsilon^{1/3} n}|\tilde\omega|$;

\item[(b)] $f^r$ maps $\omega'$ bijectively onto $\Omega_\infty^\pm$.

\end{itemize}
\end{lemma}

\begin{proof}
Let $\hat n=\min\{i\geq n\colon f^i\tilde\omega\text{ is free}\}$.

\begin{sublemma}\label{suclaim}
$\left|\left\{ S_{e(\tilde\omega)}< \hat n+\varepsilon^{1/2} n\right\}\cap\tilde\omega\cap \Omega_\infty\right|\geq
(1/3)|\tilde\omega|.$
\end{sublemma}
\begin{proof}
If the reverse inequality were true, then using
$|\tilde\omega\setminus \Omega_\infty|\leq(1/2)|\tilde\omega|$ which follows from Lemma \ref{positi} we would get
$\left|\left\{S_{e(\tilde\omega)}< \hat n+\varepsilon^{1/2} n|\tilde\omega\right\}\right|<1/2+1/3$.
This yields
a contradiction to $\left|\left\{S_{e(\tilde\omega)}\geq \hat n+\varepsilon^{1/2} n|\tilde\omega\right\}\right|\leq C\zeta^n$ obtained from Lemma \ref{escape} provided
$n$ is sufficiently large.
\end{proof}

By Sublemma \ref{suclaim} and $\#\left\{ \omega'\in
\mathcal F_{e(\tilde\omega)}|\tilde\omega\colon  S_{e(\tilde\omega)}(\omega')< \hat n+\varepsilon^{1/2} n\right\}
\leq \sum_{i=0}^{[\varepsilon^{1/2}n]}2^{i+1}\leq 3^{\varepsilon n}, $ 
one can choose an integer $r\in[\hat n,\hat n+\varepsilon^{1/2} n]\subset[n,(1+\varepsilon^{1/3})n]$ and
$\omega_0\in\mathcal F_{e(\tilde\omega)}|\tilde\omega$ intersecting $\Omega_\infty$ such that $S_{e(\tilde\omega)}(\omega_0)=r$ and
$|\omega_0|
\geq(1/3)3^{-\varepsilon^{1/2} n}|\tilde\omega|.$
Define $\omega'\subset\omega_0$ by
 $\omega'=(f^{r}|\omega_0)^{-1}\Omega_\infty^\pm$
if $f^{r}\omega_0=\Lambda^\pm$, respectively.
Since $\omega_0\cap\Omega_\infty\neq\emptyset$,
$\omega_0\subset\Omega_r$. This and $f^{r}\omega_0=\Lambda^\pm$ together 
imply $\omega_0\subset\Omega_\infty$.
The bounded distortion 
in Lemma \ref{izo} and Lemma \ref{positi} yield
$|\omega'|\geq(C_\varepsilon^{-1}/2)|\omega_0|\geq e^{-\varepsilon^{1/3}n}|\tilde\omega|.$
Hence (a) holds. (b) is obvious from the construction.
\end{proof}

\subsection{Construction of a horseshoe}\label{horseshoe}
 Let $L_1,\ldots,L_q$ be a collection of pairwise disjoint closed intervals in $[f^20,f0]$ and $m$ a positive integer. We say $\{L_i\}_{i=1}^q$ \emph{generates a horseshoe for} $f^m$ if $f^m$ maps each $L_i$ ($1\leq i\leq q$) diffeomorphically onto the same interval containing $\bigcup_{i=1}^qL_i$ in its interior.

\begin{lemma}\label{horse}
For all large $n$ there exist a
 collection $L_1,\ldots,L_q$ of closed intervals and 
an integer $m\in[(1-\varepsilon^{1/5})n,(1+\varepsilon^{1/5}) n]$ 
 such that:

\begin{itemize}
\item[(a)] $\{L_i\}_{i=1}^q$ generates a horseshoe for $f^m$;

\item[(b)] $\sum_{i=1}^q |L_i|\geq 
e^{-\varepsilon^{1/5} n}\hat\mu(\mathcal B_n')$;

\item[(c)] for all 
$x\in \bigcap_{j=0}^\infty(f^m)^{-j}\left(\bigcup_{i=1}^q L_i\right)$, 
$(1/m)S_m\varphi_j(x)\geq
b_j-\varepsilon^{1/2}$, $j=1,\ldots,d$.
\end{itemize}
\end{lemma}

\begin{proof}
Let $A\in\mathcal B_n'$ and $l_A\geq 0$ be such that 
$A\subset\Delta_{l_A}$. In the first $n$-iterates under $\hat f$, the set $A$ continues
climbing the tower, or else falls down from the tower several times.
Hence, there exists an integer $k_A\in[0,n-1]\cup\{n+l_A\}$ such that $\hat
f^nA\in\mathcal D|\Delta_{k_A}$. Thus $p_1(\hat f^nA)$ is an element of $\mathcal
P_{k_A}$, which we denote by $\omega_A$. 
If $k_A\geq n_0''$, then
take $\omega'_A\subset
\tilde\omega_A\in\tilde{\mathcal P}_{k_A}$ for which the
conclusions of Lemma \ref{Mar} hold.
Define an interval $\tilde A$ containing $p_1A$ so that:
$\tilde A=\tilde \omega_A$ if $k_A=n+l_A$;
 $f^{n-k_A+l_A}\tilde A=\tilde\omega_A$ if $n_0''\leq k_A\leq n-1$;
$f^{n-k_A+l_A}\tilde A=\Lambda^\pm$ if $k_A<n_0''$.
Set
$$\ell_A=\min\{j\geq l_A\colon \text{$f^j\tilde A$ is
free}\},$$
and define
\begin{equation*}
t_A=\begin{cases}
n-k_A+r(\omega'_A)-\ell_A+l_A&\ \text{if}\ k_A\geq n_0'';\\
n-k_A-\ell_A+l_A&\ \text{if}\ k_A<n_0''.\end{cases}\end{equation*}

\begin{sublemma}\label{tA}
For any $A\in\mathcal B_n',$
$\left(1-\varepsilon^{1/2}\right)n\leq t_A\leq
\left(1+\varepsilon^{1/4}\right)n.$
\end{sublemma}
\begin{proof}
Lemma \ref{Mar}
gives $k_A\leq
r(\omega'_A)\leq\left(1+\varepsilon^{1/3}\right)k_A.$
By construction and Lemma \ref{reclem1}(b),
$l_A\leq \ell_A\leq\left(1+
6\varepsilon\right)l_A.$ Hence, if $k_A\geq n_0''$ then
$$t_A\leq n-k_A+r(\omega_A')\leq n+\varepsilon^{1/3}k_A\leq (1+31\varepsilon^{1/3})n<(1+\varepsilon^{1/4})n,$$
where the third inequality follows from $k_A\leq n+l_A$ and $l_A\leq 30n$.
On the other hand,
$$t_A\geq n-6\varepsilon l_A\geq\left(1-180\varepsilon\right)n\geq(1-\varepsilon^{1/2})n.$$ 
If $k_A<n_0''$, then clearly
$t_A\leq n$, and
$$t_A\geq n-k -6\varepsilon l_A\geq (1-180\varepsilon)n-n_0''
\geq(1-\varepsilon^{1/2})n,$$
where  the last inequality holds provided $n$ is sufficiently large. 
\end{proof}

By construction, for any $A\in\mathcal B_n'$ one can choose a set $A'\subset\Delta_{l_A}$ 
so that: $p_1(\hat f^{n}A')=\omega_A'$
if $k_A=n+l_A$;
$p_1(\hat f^{n-k_A}A')=\omega_A'$
if $n_0''\leq k_A\leq n-1$;
$A\subset A'$
and $p_1(\hat f^{n-k_A}A')=\Omega_\infty^\pm$
if $k_A< n_0''$.


\begin{sublemma}\label{estimate}
For any $A\in\mathcal B_n'$,
$\hat\mu(A')\geq e^{-\varepsilon^{1/4}n}\hat\mu(A).$
\end{sublemma}
\begin{proof}
Lemma \ref{mhat} 
gives
$$\frac{\hat \mu(A')}{\hat \mu(A)}\geq C_1C_2^{-1}\frac{|p_1A'|}{|p_1A|}.$$
As for the fraction of the right-hand-side, if $k_A= n+l_A$, then
using $\omega_A\subset\tilde\omega_A$ and 
Lemma \ref{Mar} 
we have
$$\frac{|p_1A'|}{|p_1A|}= \frac{|\omega_A'|}{|\omega_A|}
\geq \frac{|\omega_A'|}{|\tilde\omega_A|}\geq e^{-\varepsilon^{1/3}(n+l_A)}\geq
e^{-31\varepsilon^{1/3}n}.
$$
If $n_0''\leq k_A\leq n-1$, then additionally using
the bounded distortion we have
$$\frac{|p_1A'|}{|p_1A|}
\geq C_\varepsilon^{-1}\frac{|p_1(\hat f^{n-k}A')|}{|p_1(\hat f^{n-k}A)|}=
C_\varepsilon^{-1} \frac{|\omega_A'|}{|\omega_A|}
\geq C_\varepsilon^{-1} \frac{|\omega_A'|}{|\tilde\omega_A|}\geq
 C_\varepsilon^{-1} e^{-31\varepsilon^{1/3}n}.
$$
If $k_A<n_0''$, then 
$\frac{|p_1A'|}{|p_1A|}\geq 1$ because $A\subset A'$.
Consequently the desired inequality holds provided $n$ is sufficiently large.
\end{proof}

Returning to the proof of Lemma \ref{horse}, choose 
$m_0\in[(1-\varepsilon^{1/2})n,
(1+\varepsilon^{1/4})n]$ such that
\begin{equation}\label{horse1}
\sum_{A\in\mathcal B_n':t_A=m_0}\hat \mu(A')\geq
\frac{1}{2\varepsilon^{1/4} n}\sum_{A\in\mathcal B_n'}\hat \mu
(A').\end{equation} 
Set $q=\#\{A\in\mathcal B_n'\colon
t_A=m_0\}$ and $\{A_i\}_{i=1}^q=\{A\in\mathcal B_n'\colon
t_A=m_0\}$. For each $i\in[1,q]$ let
$\hat K_i$ denote the smallest closed
interval containing $p_1A_i'$, and define
$K_i=f^{\ell_{A_i}}\hat K_i$.
Then $K_i\subset (f^20,f0)$, and 
$f^{m_0}K_i=\Lambda^\pm$ because $f^{t_A}(f^{\ell_A}A')=\Omega_\infty^\pm$
for any $A\in\mathcal B_n'$. 
By (A4) it is possible to choose $m_1>0$ and
two closed intervals $I^\pm\subset\Lambda^\pm$ such that $f^{m_1}$ maps $I^\pm$ diffeomorphically onto the same interval containing $K_1,\ldots,K_q$. 
Define
$L_i=(f^{m_0}|K_i)^{-1}I^\pm$ if
$f^{m_0}K_i=\Lambda^\pm$, respectively.
By construction, $L_1,\ldots,L_q$ are pairwise disjoint. Then
$m=m_0+m_1\in[(1-\varepsilon^{1/5})n,(1+\varepsilon^{1/5})n]$ holds for sufficiently large $n$,
and $\{L_i\}_{i=1}^q$ generates a horseshoe for $f^{m}$.
In addition,
by Lemma \ref{mhat}, \eqref{horse1} and Sublemma \ref{estimate},
\begin{equation}\label{upper31}\sum_{A\in\mathcal B_n':t_A=m_0}\hat \mu(A')=
\sum_{i=1}^q\hat \mu(A_i')\geq
\frac{1}{2\varepsilon^{1/4} n}e^{-\varepsilon^{1/4} n}\hat\mu(\mathcal B_n')
\geq e^{-\varepsilon^{1/5}n}\hat\mu(\mathcal B_n').
\end{equation}
Since both
$f^{\ell_{A_i}}\hat K_i$ and $\hat K_i$ are free,
$|Df^{\ell_{A_i}}|\geq\delta$ on $\hat K_i$. Hence
$|K_i|=|f^{\ell_{A_i}}\hat K_i|\geq \delta |\hat K_i|\geq\delta |p_1A_i'|\geq C_2^{-1}\delta\cdot\hat \mu(A_i').$ Using this and the bounded distortion we get 
$|L_i|\geq C_\varepsilon^{-1}\frac{|I^+|}{|\Lambda^+|}|K_i|\geq C_\varepsilon^{-1}\frac{|I^+|}{|\Lambda^+|}C_2^{-1}\delta\cdot\hat \mu(A_i').$
Plugging this estimate into the left-hand-side of 
\eqref{upper31} yields Lemma \ref{horse}(b).
\medskip

For the proof of Lemma \ref{horse}(c) it suffices to show
$S_m\hat\varphi_j(x)\geq (b_j-\varepsilon^{1/2})m$ for all
$x\in \bigcup_{i=1}^q (\hat f)^{l_i-l_{A_i}}A_i$.
Pick $x_i\in
A_i\in\mathcal B_n$ such that $S_n\hat\varphi(x_i)\geq b_jn$ holds for
$j=1,\ldots,d$. We have
\[
|S_m\hat\varphi_j(\hat
f^{l_i-l_{A_i}}x_i)-S_n\hat\varphi_j(x_i)|\leq
(2(l_i-l_{A_i})+|m-n|)\|\varphi_j\|,\]
 where
$\|\varphi_j\|=\sup|\varphi_j|$. 
Since $|l_i-l_{A_i}|\leq \frac{3\varepsilon}{\lambda} l_{A_i}\leq \varepsilon^{2/3} n$ and
$|m-n|\leq \varepsilon^{2/3} n$ we have
$$S_m\hat\varphi_j(\hat
f^{l_i-l_{A_i}}x_i)\geq S_n\hat\varphi_j(x_i)-
(2(l_i-l_{A_i})+|m-n|)\|\varphi_j\|\geq b_jn-2\varepsilon^{2/3}n.$$ Hence, for each $i=1,\ldots,q$ and
$j=1,\ldots,d$,
\begin{equation}\label{con1}S_m\hat\varphi_j(\hat f^{l_i-l_{A_i}}x_i)\geq
\left(b_j-\varepsilon^{1/2}/2\right)m.\end{equation} 
\begin{sublemma}\label{bdd}
For any $n\geq1$ and
$\omega\in\tilde{\mathcal P}_{n-1}$,
$\sum_{i=0}^{n-1}|f^i\omega|\leq 10\delta^{-1}.$
\end{sublemma}

\begin{proof}
Let $n_1<\cdots<n_s<n$ denote all the free returns in the first
$n-1$-iterates of $\omega$, with $p_1,\ldots,p_{s}$ the
corresponding bound periods. 
Let $1\leq i\leq s$. For each $j\in[n_i+1,n_i+p_i-1]$, choose 
$\theta_j\in f^{n_i}\omega$ such that $|f^j\omega|=|f^{n_i}\omega|\cdot
|Df^{j-n_i}(\theta_j)|.$ Then $|f^{j-n_i}\theta_j-f^{j-n_i}0|\leq
e^{-\varepsilon(p_i-1)},$ and by the bounded distortion during the bound
period,
$$|Df^{j-n_i-1}(f\theta_j)|
\leq 2\cdot\frac{|f^{j-n_i}\theta_j-f^{j-n_i}0|}{|f\theta_j-f0|}
\leq \delta_{p_i}^{-2} e^{-\varepsilon (p_i-1)}
\leq 3 \delta_{p_i-1}^{-2} e^{-\varepsilon(p_i-1)}e^{\alpha\sqrt{p_i}}.$$ 
For the last
inequality we have used (\ref{size}). We also have
$|Df(\theta_j)|\leq 4\delta_{p_i-1}.$ Plugging these two derivative
estimates into the equality and summing the result over all $j$
gives
$$\sum_{j=n_i+1}^{n_i+p_i-1} |f^j\omega|
\leq |f^{n_i}\omega| \delta_{p_i-1}^{-1}  e^{-\frac{\varepsilon}{2}(p_i-1)} .$$
Summing this over all $i$ gives
\begin{align*}
\sum_{i=1}^s\sum_{j=n_i+1}^{n_i+p_i-1}|f^j\omega|
\leq\sum_{i=1}^s|f^{n_i}\omega| 
\delta_{p_i-1}^{-1} e^{-\frac{\varepsilon}{2}(p_i-1)} 
\leq\sum_{p\geq N}
\delta_{p-1}^{-1}e^{-\frac{\varepsilon}{2}(p-1)}
\sum_{i\colon p_i=p}|f^{n_i}\omega|.\end{align*} 
Let $n_{i_j}$, $j=1,\ldots,t$ denote
the subsequence of returns with the same bound period equal to $p$. By Lemma
\ref{reclem1} and Lemma \ref{exp2},  for all $\theta\in
f^{n_{i_j}}\omega$ we have $|Df^{n_{i_t}-n_{i_j}}(\theta)|\geq
e^{\frac{\lambda p}{3}(t-j)}$, and thus
$|f^{n_{i_j}}\omega|\leq e^{-\frac{\lambda p}{3}(t-j)}
|f^{n_{i_t}}\omega|$. We also have $|f^{n_{i_t}}\omega|\leq 2\delta_{p-1}$, 
and therefore
\[\sum_{i\colon
p_i=p}|f^{n_i}\omega|=\sum_{j=1}^t|f^{n_{i_j}}\omega|\leq
\sum_{j=1}^t e^{-\frac{\lambda p}{3}(t-j)}|f^{n_{i_t}}\omega| \leq
2|f^{n_{i_t}}\omega|\leq 4\delta_{p-1}.\]
Substituting
this estimate into the previous inequality gives
\begin{equation*}\label{sum2}\sum_{j\in\cup_{i=1}^s(n_i,n_i+p_i)}|f^j\omega|
\leq \sum_{p\geq N}e^{-\frac{\varepsilon}{2}(p-1)}.\end{equation*}
We use part of the estimates in the proof of Lemma \ref{izo}
to get
\begin{equation*}\label{sum1}
\sum_{j\in[0,n)\setminus\cup_{i=1}^s(n_i,n_i+p_i)}|f^j\omega|\leq 5\delta^{-1}|f^{n}\omega|\leq10\delta^{-1}.
\end{equation*}
These two inequalities yield the desired one.\end{proof}
By Sublemma \ref{bdd}, for any $x\in \hat
f^{l_i-l_{A_i}}A_i$ we have
$|S_n\hat\varphi_j(\hat f^{l_i-l_{A_i}}x_i)-S_n\hat\varphi_j(x)|\leq
{\rm Lip}(\varphi_j)\cdot 10\delta^{-1}$,
where ${\rm Lip}(\varphi_j)$ denotes the Lipschitz constant of $\varphi_j$.
Hence we have
\begin{equation}\label{con2}
|S_m\hat\varphi_j(\hat
f^{l_i-l_{A_i}}x_i)-S_m\hat\varphi_j(x)|\leq
{\rm Lip}(\varphi_j)\cdot 10\delta^{-1}+2\|\varphi_j\|(m-n)\leq
(\varepsilon^{1/2}/2)m.\end{equation}
From (\ref{con1}), (\ref{con2}) we obtain $S_{m}\hat\varphi_j(x)\geq
(b_j-\varepsilon^{1/2})m$. 
\end{proof}

\subsection{Construction of a measure on the horseshoe}\label{meas}
We construct a measure for which (\ref{upper5})
(\ref{upper3}) hold under the assumption that $\mathcal B_n'\neq\emptyset$.  
Let $L_1,\ldots,L_q$ be a collection of pairwise disjoint closed intervals
and $m$ a positive integer for which 
the conclusions of Lemma \ref{horse} 
hold.
Set $H=\bigcap_{j=0}^\infty(f^m)^{-j}\left(\bigcup_{i=1}^q L_i\right)$
and define $g=f^m|H$. 
The Koebe Principle implies that there exist constants $c>0$ and $\kappa>1$
such that for any $x\in H$ and every $n\geq0$, $|Dg^{n}(x)|\geq c\kappa^n$.
This implies that $g\colon
H\circlearrowleft$ is H\"older conjugate to
the one-sided full shift on $q$-symbols.
Define a continuous function $\Phi\colon H\to\mathbb R$ by
$\Phi(x)=\log |Dg(x)|.$
Pick an equilibrium state of $g$ for the potential $-\Phi$ and denote it by $\nu_\Phi$.
Namely, $\nu_\Phi$ is a $g$-invariant probability measure and satisfies
$$h_g(\nu_\Phi)-\nu_{\Phi}(\Phi)=\sup\left\{h_g(\nu)-\nu(\Phi)
\colon\text{$\nu$ is $g$-invariant}\right\}.$$ Here, $h_g(\nu)$
denotes the entropy of $(g,\nu)$.
Let $\sigma=(1/m)\sum_{i=0}^{m-1}(f^i)_*\nu_{\Phi}$, which is
$f$-invariant and ergodic. From Lemma \ref{horse}(c) it
follows that $S_m\varphi_j\geq (b_j-\varepsilon^{1/2})m$
$\nu_{\Phi}$-a.e. Hence $\sigma(\varphi_j)
=(1/m)\nu_{\Phi}(S_{m}\varphi_j)\geq b_j-\varepsilon^{1/2}$, and
(\ref{upper5}) holds.

For $k>0$ and a $(k+1)$-string
$(a_0,\ldots,a_k)$ of integers in $[1,q]$, let $$L_{a_0\cdots
a_{k}}=L_{a_0}\cap g^{-1}L_{a_1}\cap\cdots \cap g^{-k}L_{a_k}.$$
By the Koebe Principle, there exists $\tau\in(0,1)$ such that 
$|L_{a_0\cdots a_{k-1}a_{k}}|/|L_{a_0\cdots
a_{k-1}}| \geq \tau |L_{a_{k}}|.$ Hence
\begin{align*}
\sum_{(a_0,\ldots, a_k)} |L_{a_0\cdots a_k}|&=
\sum_{(a_0,\ldots, a_{k-1})}|L_{a_0\cdots
a_{k-1}}|\sum_{a_k}\frac{|L_{a_0\cdots a_{k-1}a_k}|}
{|L_{a_0\cdots a_{k-1}}|}\\
&\geq \tau\sum_{j=1}^q|L_j|\sum_{(a_0,\ldots, a_{k-1})}|L_{a_0\cdots
a_{k-1}}|\geq \left(\tau\sum_{j=1}^q|L_j|\right)^{k+1}.\end{align*}
This yields
\begin{equation}\label{lem1}\varliminf_{k\to\infty}\frac{1}{k}
\log\sum_{(a_0,\ldots, a_{k})}|L_{a_0\cdots
a_{k}}|\geq\log\sum_{j=1}^q|L_j|+\log \tau .\end{equation}

Let 
$\nu_{a_0\cdots a_k}$ denote the
atomic probability measure equally distributed on the periodic orbit of $g$ of period $k+1$ in $L_{a_0\cdots a_k}$.
Define a $g$-invariant probability measure $\nu_k$ by
$$\nu_k=\rho_k\sum_{(a_0,\ldots, a_k)}
|L_{a_0\cdots a_k}|\cdot\nu_{a_0\cdots a_k},$$ where $\rho_k$ is the
normalizing constant. Pick an accumulation point of the sequence
$\{\nu_k\}_k$ and denote it by $\nu_\infty$. Taking a subsequence
if necessary we may assume this convergence takes place for the
entire sequence. By the relation
 $\nu_k(L_{a_0\cdots a_k})=\rho_k|L_{a_0\cdots a_k}|$
and $|L_{a_0\cdots a_k}|\leq
\tau^{-1}e^{-(k+1)\nu_{a_0\cdots a_k}(\Phi)}$ we have
\begin{align*}
\log\sum_{(a_0,\ldots, a_k)}|L_{a_0\cdots a_k}|
&=\sum_{(a_0,\ldots, a_k)}\nu_k(L_{a_0\cdots
a_k})\left(-\log\nu_k(L_{a_0\cdots a_k})+
\log|L_{a_0\cdots a_k}|\right)\\
&\leq-\sum_{(a_0,\ldots, a_k)}\nu_k(L_{a_0\cdots
a_k})\log\nu_k(L_{a_0\cdots a_k})-(k+1)\nu_k(\Phi)-\log \tau .
\end{align*}
Then the usual proof of the variational principle \cite[Theorem 9.10]{Wal82} shows
\begin{equation}\label{lem2}
h_g(\nu_\infty)-\nu_\infty(\Phi)\geq\varlimsup_{k\to\infty}\frac{1}{k}\log\sum_{(a_0,\ldots,
a_k)}|L_{a_0\cdots a_k}|.
\end{equation}
Combining (\ref{lem1})
(\ref{lem2}) and then using Lemma \ref{horse}(b), for all large $n$ we have  \begin{equation*}h_g(\nu_\infty)-\nu_\infty( \Phi)
\geq \log\sum_{i=1}^q|L_i|+\log\tau \geq-2\varepsilon^{1/5} n+
\log\hat\mu(\mathcal B_n').\end{equation*} Since $F(\sigma)\leq0$ and $m\geq(1-\varepsilon^{1/5})n$ we have
\begin{align*}
F(\sigma)n&\geq \frac{F(\sigma)m}{1-\varepsilon^{1/5}}  =
\frac{h_g(\nu_{\Phi})-\nu_{\Phi}(\Phi)}{1-\varepsilon^{1/5}}\geq\frac{h_g(\nu_\infty)-\nu_\infty(\Phi)}{1-\varepsilon^{1/5}}\geq
\frac{-2\varepsilon^{1/5} n+\log\hat\mu(\mathcal B_n')}{1-\varepsilon^{1/5}}.\end{align*}
Rearranging this yields \eqref{upper3} and hence \eqref{upper1}. 
\qed


\subsection*{Acknowledgments.}
We thank Michihiro Hirayama, Toshio Mikami, Feliks Przytycki,
Juan Rivera-Letelier, Yoichiro Takahashi, Masato Tsujii and
Paulo Varandas for fruitful discussions.
The first-named author is partially supported by the Kyoto University Global COE Program.
The second-named author is supported by the Aihara Project, the FIRST Program from
the JSPS, initiated by the CSTP.

\end{document}